\documentclass[12pt, a4paper]{amsart}
\usepackage{amsfonts,amsmath, amsthm, amssymb, amscd}
\usepackage{latexsym}
\input{amssym.def}
\input{amssym}

\usepackage[dvips]{graphicx}

\newcommand{\cal}{\mathcal}

\newtheorem{theorem}{Theorem}[section]
\newtheorem{lemma}[theorem]{Lemma}
\newtheorem{cor}[theorem]{Corollary}
\newtheorem{prop}[theorem]{Proposition}
\newtheorem{dfn}[theorem]{{Definition}}

\numberwithin{equation}{section}

\newcommand {\R}{\mathbb{R}} %% reals
\DeclareMathOperator{\id}{id}
\DeclareMathOperator{\vol}{vol}

\DeclareMathOperator{\End}{End}
\DeclareMathOperator{\grad}{grad}
\DeclareMathOperator{\rank}{rank}
\DeclareMathOperator{\tr}{tr}
\def\remark{\par\medbreak\noindent{\sc Remark}\quad\enspace}

\begin{document}

\title[New results on noncompact harmonic manifolds]{New results on  noncompact harmonic manifolds}

\author{Gerhard Knieper }
\date{\today}
\address{Faculty of Mathematics,
Ruhr University Bochum, 44780 Bochum, Germany}
\email{gerhard.knieper@rub.de}

%%%%%%%%%%%%%%%%%%%%%%%%%%%%%%%%%%%%%%%%%%%%%%%%%%%%%%%%%%%%%%%%%%%%%%

\begin{abstract}
The Lichnerowicz conjecture asserts that all harmonic manifolds are either flat
or locally symmetric spaces of rank~1.
This conjecture has been proved by  Z.~Szab\'{o} \cite{Sz} for harmonic manifolds with compact universal cover.
E.~Damek and F.~Ricci \cite{DR} provided examples showing that in the noncompact case the
conjecture is wrong. However, such manifolds do not admit a compact quotient.

In this paper we study, using a notion of rank, the asymptotic geometry and the geodesic flow on simply connected nonflat and noncompact harmonic manifolds denoted by $X$.

In the first part of the paper
we show that the following assertions are equivalent. The volume growth
is purely exponential, the rank of $X$ is one, the geodesic flow is Anosov with
respect to the Sasaki metric, $X$ is Gromov hyperbolic.

In the second part of the paper we show that the geodesic flow is Anosov if $X$ is a nonflat harmonic manifold with no focal points.
In the course of the proof we obtain that certain  partially hyperbolic flows on arbitrary Riemannian manifolds without focal points are Anosov, which is of interest beyond harmonic manifolds.

Combining the results of this paper with
the rigidity theorem's of \cite{BCG} , \cite{BFL}
 and \cite{FL}, we confirm the Lichnerowicz
conjecture for all compact harmonic manifolds
without focal points or with Gromov hyperbolic fundamental groups.
\end{abstract}

%%%%%%%%%%%%%%%%%%%%%%%%%%%%%%%%%%%%%%%%%%%%%%%%%%%%%%%%%%%%%%%%%%%%%%

\maketitle

%%%%%%%%%%%%%%%%%%%%%%%%%%%%%%%%%%%%%%%%%%%%%%%%%%%%%%%%%%%%%%%%%%%%%%

\section{Introduction}
A complete Riemannian manifold $X$ is called harmonic
if the harmonic functions satisfy the mean value property, that is, the average on any sphere coincides with its value in the center.
Equivalently, for any $p \in
X$ the volume density $\theta_p(q) = \sqrt{\det g_{ij}(q)}$
in normal coordinates, centered at any point $p \in
X$ is a radial function. In particular, if $c: [0, \infty) \to X$ is a normal geodesic with $c(0) =p$,
the function $f(t) := \theta_p(c(t))$ is independent of $c$. It is easy to see that all rank~1 symmetric
spaces and Euclidean spaces (model spaces) are harmonic. In 1944, A.~Lichnerowicz conjectured that conversely every complete harmonic manifold
is a model space. He confirmed the conjecture up to dimension~4 \cite{Li}. It was not before the beginning of the 1990's that general results where
obtained. In 1990 Z.~Szab\'{o} \cite{Sz} proved the Lichnerowicz conjecture for compact simply connected spaces. However, not much later,
in 1992, E.~Damek and F.~Ricci \cite{DR} showed that in the noncompact case the conjecture is wrong. They provided examples of homogeneous
harmonic spaces which are not symmetric.
Nevertheless, in 1995 G.~Besson, G.~Courtois and S.~Gallot \cite{BCG} confirmed the conjecture for manifolds of negative curvature admitting a
compact quotient. The proof consisted in a combination of deep rigidity results from hyperbolic dynamics
and used besides \cite{BCG} the work of Y.~Benoist, P.~Foulon and F.~Labourie (\cite{BFL} and P.~Foulon and
F.~Labourie \cite{FL}).

In 2002 A.~Ranjan and H.~Shah showed \cite{RSh2} that noncompact harmonic manifolds of polynomial volume growth are flat. Using a result by
 Y.~Nikolayevski \cite{Ni} showing that the density function $f$ is a exponential polynomial the result of A.~Ranjan and H.~Shah remains true
 under the assumption of subexponential volume growth.
In 2006 J.~Heber \cite{He} proved that among the homogeneous harmonic spaces only the model spaces and the Damek-Ricci spaces occur.
Therefore, it remains to study nonhomogenous harmonic manifolds of exponential volume growth. In particular, these are spaces without
conjugate points and horospheres of constant mean curvature $h >0$.

The starting point of this paper was a question asked by N.~Peyerimhoff whether noncompact harmonic manifolds with
positive mean curvature $h >0$ of the horospheres have purely exponential volume
growth, i.e., the quotient of the density function $f(t)$ and $e^{ht}$ stays for large $t$ between two positive constants.

It turned out that the answer to this question is intimately related to the notion of rank, which is a straight forward generalization of the
 wellknown rank of manifolds of nonpositive curvature \cite{BBE}. In particular,
the volume growth is purely exponential if and only if the rank is 1.
Moreover, we show that noncompact harmonic manifolds are of  rank~1 if and only if the geodesic flow is Anosov.
Therefore, having a compact quotient the rigidity theorems mentioned above force harmonic spaces of rank~1 as in the case of negative curvature
to be locally symmetric.

We believe that all nonflat harmonic manifolds are of rank~1. We confirm this for instance for all harmonic manifolds
without focal points which include spaces of nonpositive curvature. It is very likely that all noncompact
harmonic manifolds have no focal points.
All known examples of noncompact harmonic manifolds have nonpositive curvature.
By a result of A.~Ranjan and H.~Shah \cite{RSh1} a sufficient condition for no focal points is provided under the assumption that  besides the
 trace also the determinant of the second fundamental form of the geodesic spheres is a function of the radius. \\

The paper is organized as follows. In section 2 we introduce the techniques of Jacobi tensors. In particular, we show that manifolds
of constant negative curvature have minimal volume growth
among those harmonic manifolds with fixed mean curvature of the horospheres.\\

In section 3 we introduce the notion of rank for manifolds without conjugate points. We
 show that for harmonic spaces the following three properties are equivalent: $X$ has rank~1, the geodesic flow on $X$ is Anosov, the volume
growth is purely exponential.

 In section 4 we show that  for noncompact simply connected harmonic Gromov hyperbolicity is equivalent to the three properties studied in section three. Hence, compact harmonic spaces with Gromov hyperbolic fundamental
 groups are locally symmetric.

In section 5 we study harmonic manifolds $X$ of bounded asymptote and show that the rank
is constant. In particular, they include all manifolds without focal points, i.e.,
manifolds for which geodesic spheres are convex.
Using a classical topological result of N.E.~Steenrod and J.H.C.~Whitehead on vector fields of spheres \cite{SW} we show that in odd dimensions
the rank is one. If additionally $X$ admits a compact quotient $X$ has to have constant negative curvature.

In section 6 we study geodesic flows on arbitrary manifolds without focal points and constant rank. Under the assumption of bounded sectional curvature together with a certain transversality  condition we obtain first that the geodesic
flow is partially hyperbolic. Using a geometric argument we finally show
that the flow is Anosov.
 By adding to this the results of chapter 5 we obtain that for all harmonic manifolds without focal points the geodesic flow is Anosov as well.

In the appendix we collect for the convenience of the reader properties of Jacobi tensors which are important
in this paper.

%%%%%%%%%%%%%%%%%%%%%%%%%%%%%%%%%%%%%%%%%%%%%%%%%%%%%%%%%%%%%%%%%%%%%%

\section{Volume growth in harmonic manifolds}

In this paper $X$ be will denote a complete, noncompact, simply connected harmonic manifold.
This implies that $X$ is a manifold without conjugate points and thus
by a theorem of Cartan-Hadamard the exponential map $\exp_p: T_pX \to X$ is a diffeomorphism. Moreover, $X$ is an Einstein manifold and thus
analytic (see \cite{Be}).

We briefly recall the calculus of Jacobi tensors. Let $c: I \to X$ be a
 unit speed geodesic and let $Nc$ denote the normal bundle of $c$ given by a disjoint union
 $$
 N_t(c) := \{w \in T_{c(t)}X \mid \langle w, \dot{c}(t) \rangle = 0 \}.
$$
A $(1,1)$-tensor along $c$ is  a differentiable section
$$
Y: \mathbb{R} \to \End{Nc}= \bigcup_{t \in I}\End( N_t(c)),
$$
i.e., for all orthogonal parallel vector fields $x_t$ along $c$ the covariant derivative
of $t \to Y(t) x_t$ exists. The derivative $Y'(t) \in \End( N_t(c))$ is defined by
$$
Y'(t)(x_t) = \frac{D}{dt} \left(Y(t)x_t \right).
$$
$Y$ is called parallel  if $Y'(t) = 0$ for all $t$.
If $Y$ is parallel we have $Y(t)x_t = (Y(0)x)_t$ and, therefore,
$\langle Y(t)x_t , y_t \rangle$ is constant for all parallel vector fields $x_t, y_t$ along $c$ .
 In particular, $Y$ is parallel if and only if $Y$ is a constant matrix with respect to parallel frame field
 in the normal bundle of $c$. Therefore, parallel $(1,1)$-tensors are also called constant.

The curvature tensor $R$ induces a symmetric $(1,1)$-tensor along $c$ given by
$$
R(t) w : = R(w, \dot{c}(t)) \dot{c}(t).
$$
A $(1,1)$-tensor $Y$ along $c$ is called a Jacobi tensor
if it solves the Jacobi equation
$$
Y''(t) + R(t) Y(t) = 0.
$$
If $Y, Z$ are two Jacobi tensors along $c$ the derivative of the Wronskian
$$
W(Y, Z)(t) := Y'^{\ast}(t) Z(t) - Y^{\ast}(t)Z'(t)
$$
is zero and thus $W(Y, Z)$ defines a parallel $(1,1)$-tensor.
A Jacobi tensor $Y$ along a geodesic $c: I \to X$ is called Lagrange tensor if $W(Y, Y) =0$.
The importance of Lagrange tensors comes from the following proposition.

\begin{prop} \label{jac1}

Let $Y: I \to \End{Nc} $ be a Jacobi tensor along a geodesic $c: I \to X$ which is nonsingular
for all $t \in I$. Then for $t_0 \in I$ and any other Jacobi tensor $Z$ along $c$, there exist constant
tensors $C_1$ and $C_2$ such that
$$
Z(t) = Y(t) \left( \int\limits_{t_0}^{t}( Y^{\ast} Y)^{-1}(s) ds \ C_1 + C_2 \right)
$$
for all $t \in I$.

\end{prop}
\remark The definition of the integral and a proof of this proposition is given in the appendix.

Let $SX$ denote the unit tangent bundle of $X$ with fibres
$S_pX$, $p \in X$, and, for every $v \in SX$, let $c_v: \R \to X$
denote the unique geodesic satisfying $c'(0) = v$.
Define $A_v$ to be the Jacobi tensor along $c_v$ with $A_v(0) = 0$ and
$A_v'(0) = \id$.
Then the volume of a geodesic sphere $S(p,r)$ of radius $r$ about $p$ is given by
$$
\vol S(p,r) = \int\limits_{S_pX}\det A_v(r) d \theta_p(v),
$$
where $d \theta_p(v)$ is the volume element of $S_pX$ induced by the Riemannian metric.
By definition  $X$ is harmonic if and only if the volume density $f(t) = \det A_v(t)$ does not depend on $v$.
Therefore
$$
\vol S(p,r) = \omega_{n-1} f(r),
$$
where $\omega_{n-1}$ is the volume of the sphere in the Euclidean
space $\mathbb{R}^n$.
Since
$$
\frac{(\det A_v(r))'}{\det A_v(r)} = \tr (A_v'(r) A_v(r)^{-1})
$$
is the mean curvature  of the geodesic sphere of radius $r >0$ about $\pi(v)$ in $c_v(r)$,
$X$ is harmonic if and only if the mean curvature of all spheres is a function depending only on the radius.

Of fundamental importance are the stable and unstable Jacobi tensors.
For a general complete simply connected manifold without conjugate points $X$ they are defined as follows.
For $ v \in SX$ and $r > 0$ denote by
  $ S_{v,r}$  and $U_{v,r}$ the Jacobi tensors along $c_v$ such that
 $$
  S_{v,r}(0) = U_{v,r}(0) = \id  \; \; \text{and} \; \;  S_{v,r}(r) = 0 , \; U_{v,r}(-r) = 0.
 $$
 Let
 $$
 S_v =\lim\limits_{r \to \infty}S_{v,r} \; \; \text{and} \; \ U_v =\lim\limits_{r \to \infty}U_{v,r}
 $$
be the stable and unstable Jacobi tensors. Note, that $\tr U_{v,r}'(0) =\tr (A_{\phi^{-r}v}'(r) A_{\phi^{-r}v}^{-1}(r)) $ and if $X$ is complete noncompact harmonic manifold
$\tr U_{v,r}'(0) = \frac{f'(r)}{f(r)} $ is converging to $\tr U_{v}'(0)=:h \ge 0$, where $h$  is the mean curvature of the horospheres.
Hence,
$$
\lim\limits_{r \to \infty}\frac {\log \vol S(p,r)}{r}
= \lim\limits_{r \to \infty}\frac {f'(r)}{f(r)}=h
$$
\begin{dfn}
A noncompact harmonic manifold with $h>0$ is called of purely exponential
volume growth if there are constants $0 <a \le b$ such that
$$
a e^{hr} \le f(r) \le be^{hr}
$$
for all $r \ge1$.
\end{dfn}
\begin{lemma} \label{sujac}
Let $X$ be a complete simply connected manifold without conjugate points.
 Let $c_v: \R \to X$ be a geodesic with $\dot c_v(0) = v \in SX$ and $s, r > 0$.
Then we have
$$
(U_{v,r}'(0) - S_{v,s}'(0))^{-1} = \int\limits_{0}^s(U_{v,r}^{\ast}U_{v,r})^{-1}(u) du
$$
and
 $$
(U_{v}'(0) - S_{v,s}'(0))^{-1} = \int\limits_{0}^s(U_{v}^{\ast}U_{v})^{-1}(u) du.
$$
Similarly for $0 <s <r$  we have
$$
(S_{v,r}'(0) - S_{v,s}'(0))^{-1} = \int\limits_{0}^s(S_{v,r}^{\ast}S_{v,r})^{-1}(u) du
$$
and
$$
(S_{v}'(0) - S_{v,s}'(0))^{-1} = \int\limits_{0}^s(S_{v}^{\ast}S_{v})^{-1}(u) du.
$$
Furthermore, the function
$$
\det  \left(\int\limits_{0}^s(U_{v}^{\ast}U_{v})^{-1}(u) du\right) = \frac{1}{\det (U_{v}'(0) - S_{v,s}'(0))}
$$
is strictly monotonically increasing.
\end{lemma}
\begin{proof}
Let $s,r$ be positive real numbers. For all $s > -r$ the endomorphism $U_{v,r}(s)$ is nonsingular and Lagrangian.
Using proposition \ref{jac1} we obtain for all $t > -r$
$$
S_{v,s}(t) = U_{v,r}(t) \int\limits_{t}^s(U_{v,r}^{\ast}U_{v,r})^{-1}(u) \ du \;C_{r,s}
$$
for a constant $(1,1)$- tensor $C_{r,s}$. Evaluating and differentiating this identity at $t = 0 $ yields
$$
\id = S_{v,s}(0) = \int\limits_{0}^s(U_{v,r}^{\ast}U_{v,r})^{-1}(u) \ du \; C_{r,s}
$$
and
\begin{eqnarray*}
S_{v,s}'(0) &=& U_{v,r}'(0)\int\limits_{0}^s(U_{v,r}^{\ast}U_{v,r})^{-1}(u) \ du \; C_{r,s} - C_{r,s}\\
&=& U_{v,r}'(0) - C_{r,s}
\end{eqnarray*}
which proves the first equation. Taking on both sides the limit $r \to \infty$ yields the second
equation. \\
Now consider $0 <s <r$. Again using proposition \ref{jac1}  we obtain for all $t < r$
$$
S_{v,s}(t) = S_{v,r}(t) \int\limits_{t}^s(S_{v,r}^{\ast}S_{v,r})^{-1}(u) \ du \;D_{r,s}
$$
for a constant $(1,1)$-tensor $D_{r,s}$. As above evaluating and differentiating this identity at $t = 0 $ yields the second assertion.
Since
$$
0 < \langle (U_{v}'(0) - S_{v,s_2}'(0))x,x \rangle < \langle( U_{v}'(0) - S_{v,s_1}'(0))x,x \rangle
$$
for $s_1 < s_2$ and $x \in v^\perp$, we obtain the last claim.
\end{proof}
\begin{prop}
Let $X$ be a noncompact, simply connected harmonic manifold and $h = 0$.
Then $X$ is flat.
\end{prop}
\begin{proof}
>From \cite{Ni} follows that $X$ has polynomial volume growth. But this implies by a result of A.~Ranjan and H.~Shah \cite{RSh2} that $X$ is flat.
\end{proof}
\begin{cor} \label{volgr}
Let $X$ be a noncompact, simply connected harmonic manifold such that $h >0$.
Then the function $F: [0, \infty) \to [0, \infty)$ given by
$$
F(t) = \frac{f(t)}{e^{ht}} = \frac{1}{\det (U_{v}'(0) - S_{v,t}'(0))}
$$
is strictly monotonically increasing. Moreover,
$$
\lim\limits_{t \to \infty} F(t)=
\left\{ \begin{array}{ccr}
\infty, & \text { if\ }& \det (U_{v}'(0) - S_{v}'(0)) =0
\\
 \frac{1}{\det (U_{v}'(0) - S_{v}'(0))}, & \text{if \ }& \det(U_v'(0) -S_v'(0)) >0
\end{array} \right.
$$
In particular,
$
\det(U_v'(0) -S_v'(0))
$
and $\det (U_{v}'(0) - S_{v,t}'(0))$
are independent of $v \in SX $ and
$$
a e^{ht}\le f(t)
$$
for all $t \ge 1$, where $a = \frac{1}{\det (U_{v}'(0) - S_{v,1}'(0))}$
\end{cor}
\begin{proof}
Let $A_v$ be the Jacobi tensor along the geodesic $c_v: \R \to X$  with $\dot c_v(0) = v \in SX$
such that $A_v(0) = 0$ and $A_v'(0) = \id$. Then proposition \ref{jac1} implies
$$
A_v(t)= U_{v}(t) \int\limits_{0}^t(U_{v}^{\ast}U_{v})^{-1}(u) du.
$$
Since by lemma \ref{central} we have
$$
(\log \det U_v)'(t) = \tr U_v'(t)  U_v^{-1}(t) = \tr U_{\phi^t(v)}'(0) =h,
$$
we obtain
\begin{eqnarray*}
\frac{f(t)}{e^{ht}} &= &\frac{\det A_v(t)}{ \det U_v(t)} = \det \left(\int\limits_{0}^t(U_{v}^{\ast}U_{v})^{-1}(u) du\right)\\
&= & \frac{1}{\det (U_{v}'(0) - S_{v,t}'(0))}.
\end{eqnarray*}
In particular, $\det (U_{v}'(0) - S_{v,t}'(0))$ as well as $\det (U_{v}'(0) - S_{v}'(0))$ are independent of $v \in SX$.

\end{proof}
Using the result above we obtain that  manifolds
of constant negative curvature have minimal volume growth
among those harmonic manifolds with fixed mean curvature of the horospheres. More precisely:
\begin{cor}
Let $X$  be a $n$-dimensional, noncompact, simply connected harmonic manifold with mean curvature of the horospheres equal to $h >0$. Then
$$
\lim\limits_{t \to \infty} \frac{f(t)}{e^{ht}} \ge \left( \frac{n-1}{2h} \right)^{n-1}
$$
and equality holds if and only if $X$ has constant negative sectional curvature.
\end{cor}
\begin{proof}
Note, that for a given symmetric matrix $B$ on $\mathbb{R}^k$ with positive eigenvalues, we have
$(\det B)^{1/k} \le \frac{\tr B}{k}$, where equality holds if and only if $B = \lambda \id$.
Applying this to $B =( U_v'(0) -S_v'(0))$, we obtain from the theorem above that
$$
\lim\limits_{t \to \infty} \frac{f(t)}{e^{ht}} \ge \left( \frac{n-1}{2h} \right)^{n-1},
$$
where equality holds if and only if $( U_v'(0) -S_v'(0)) = \frac{2h}{n-1}\id$.
 Let's assume that equality holds. Consider $U(v) = U_v'(0)$ and $S(v) = S_v'(0)$ then they are both solutions of the Riccati equation. Subtracting the associated two Riccati equations, we obtain
$$
0 = U'(v)- S'(v) + U^2(v) - S^2(v) = U^2(v) - S^2(v)
$$
and hence,
$$
 U^2(v) = \left(\frac{2h}{n-1}\id + S(v) \right)^2  = \left(\frac{2h}{n-1}\right)^2\id + \frac{4h}{n-1}S(v) + S(v)^2.
$$
 Since $S^2(v) = U^2(v)$, this implies
 $$
 S(v)= -\frac{h}{n-1} \id
 $$
 and, therefore,
 $$
 R_v = - S(v)^2 =  -\left(\frac{h}{n-1}\right)^2 \id,
 $$
 where $R_v$ is the Jacobi operator given by $R_v(x) = R(x,v)v$. Hence, the sectional curvature is constant.
 \end{proof}

 \section{The rank of a harmonic manifold }
 The notion of rank has been introduced for general spaces of nonpositive curvature by
 Ballmann, Brin and Eberlein \cite{BBE} and is one of the central concepts in rigidity theory. This notion can be easily generalized to manifolds without
 conjugate points.
 \begin{dfn}
Let $M$ be a manifold without conjugate points.
The rank of $v \in SM$ is defined by
$$
\rank(v) = \dim \cal L (v) +1.
$$
where
$\cal L (v) = \ker (U_v'(0) -S_v'(0))$.
 The rank of $M$
is defined to be
$$
\rank(M) = \min \{  \rank(v) \mid v \in SM \}.
$$
\end{dfn}
>From corollary \ref{volgr} we immediately obtain.
\begin{cor}
Let $X$ be a noncompact simply connected harmonic manifold.
Then $X$ has purely exponential volume growth if an only if
the rank of $X$ is one.
\end{cor}
 The following lemma is wellknown.

 \begin{lemma} \label{3.1A}
 Let $M$ be a manifold without conjugate points whose sectional curvature is bounded from
 below by $- \beta^2$ for some $\beta \ge 0$. Then
 $$
| \langle U_v'(0)x, x \rangle | \le \beta \langle x, x \rangle \;\; \text{and}\; \; | \langle S_v'(0)x, x \rangle | \le \beta  \langle x, x \rangle
 $$
 for all $v\in SM$ and $x \in v^\perp$.
 \end{lemma}
\begin{proof}
For a proof see for instance \cite{Kn}.
\end{proof}

We also will need the following result of J.~Bolton \cite{Bo} which provides a sufficient condition
for a manifold without conjugate points that their geodesic flow is Anosov.
In the compact case this result has been obtained by P.~Eberlein \cite{Eb}.
In the noncompact case one has to specify a metric in order to define the Anosov condition.
A natural metric is the Sasaki metric.  Using the isomorphism
$$
(d\pi_v, C_v) : T_v TM
\to T_{\pi v} M \times  T_{\pi v} M \; \xi \to (d\pi_v(\xi) , C_v(\xi)) = (\xi_1, \xi_2),
$$
where $\pi: TM \to M$ is the canonical projection and $C_v: T_v TM \to  T_{\pi v} M $
is the connection map, one defines the Sasaki metric via
$$
\langle \xi, \eta \rangle := \langle \xi_1, \eta_1 \rangle + \langle \xi_2, \eta_2 \rangle.
$$
Then the geodesic flow  $\phi^t SM \to SM$ is Anosov with respect to the Sasaki metric if there exists a splitting
$$
T_vSM = E^s(v) \oplus E^u(v) \oplus E^c(v)
$$
and constants $a \ge 1$ and $b >0 $ such that for all $\xi \in E^s(v)$
$$
\| D\phi^t(v) \xi \| \le a \| \xi \| e^{-bt}, \; t \ge 0 \;\;  \text{and} \;\;  \| D\phi^t(v) \xi\| \ge \frac{1}{a} \| \xi \| e^{-bt}, \; t \le  0
$$
and for all $\xi \in E^u(v)$
$$
\| D\phi^t(v) \xi \| \ge \frac{1}{a} \| \xi \| e^{bt}, \; t \ge 0 \; \;\text{and}\; \; \| D\phi^t(v) \xi \| \le a \| \xi \| e^{bt}, \; \; t \le  0.
$$
\begin{theorem} \label{Bolt}
Let $M$ be a manifold without conjugate points and sectional curvature bounded from below. Then the geodesic flow  $\phi^t: SM \to SM$ is Anosov
if and only if there exists a constant $\rho >0$ such that
$$
\langle (U_{v}'(0) - S_{v}'(0))x , x \rangle \ge \rho \langle x, x \rangle
$$
for all $x \in v^\perp$.
\end{theorem}
We recall that W.~Klingenberg \cite{Kl} and R.~Ma\~{n}\'{e} \cite{Ma} (in a more general setting) proved that Riemannian metrics on compact manifolds do not have conjugate points if their geodesic flow is Anosov.\\

\begin{theorem} \label{Anosov}
Let $X$ be a noncompact simply connected harmonic manifold. Then the geodesic flow $\phi^t: SM \to SM$ is Anosov if and only if $\rank(X) =1$.
curvature.
\end{theorem}
\begin{proof}
Assume that $\rank(X) =1$. Since by corollary~\ref{volgr} the determinant of $(U_{v}'(0) - S_{v}'(0))$
is independent of $v \in SX$, we have $\rank(v) =1$ for all $v \in SX$. By
proposition~6.57 in \cite{Be} the sectional curvature of
  a harmonic manifold is bounded. Therefore, lemma~\ref{3.1A} implies that
the eigenvalues of the nonnegative endomorphism $(U_{v}'(0) - S_{v}'(0))$ are uniformly bounded from above.
Since $\det(U_v'(0) -S_v'(0)) = const >0$  the smallest eigenvalue of $(U_v'(0) -S_v'(0))$ is bounded
from below. Hence,  we conclude from theorem~\ref{Bolt} that the geodesic flow is Anosov.
The inverse assertion is an immediate consequence of ~\ref{Bolt}.
\end{proof}

The Anosov condition on harmonic manifolds admitting a compact quotient becomes particulary interesting if we combine it  with the rigidity  of \cite{BCG} together with \cite{BFL} and \cite{FL}.
\begin{theorem} \label{BCG}
Let $(M,g)$ be a compact Riemanian manifold such that the geodesic flow is Anosov. Assume that the mean curvature of the horospheres is constant. Then $(M,g)$ is isometric to a locally symmetric space $(M_0, g_0)$ of negative curvature.
\end{theorem}
\begin{proof}
>From the work of P.~Foulon and F.~Labourie \cite{FL} follows that the stable and unstable distribution $E^s$ and $E^u$ of a geodesic flow are $C^\infty$ provided the mean curvature of the horospheres is constant. The results of
Y.~Benoist, P.~Foulon and F.~Labourie imply that the geodesic flow on the unit tangent bundle of $(M,g)$ is smoothly conjugate to the geodesic flow on the unit tangent bundle of
a locally symmetric space $(M_0, g_0)$ of negative curvature. Furthermore, $M, M_0$ are homotopy equivalent and the topological entropy as well as  the volume of both
manifolds  $(M,g)$ and $(M_0, g_0)$ coincide.
Since by a result of A.~Freir\'{e} and R.~Ma\~{n}\'{e} \cite{FM} the volume entropy and the topological entropy for
metrics without conjugate points coincide, the work of G.~Besson, G.~Courtois and S.~Gallot implies
that $(M,g)$ and $(M_0, g_0)$ are isometric.
\end{proof}
We immediately obtain.
\begin{cor}
Let $X$ be a noncompact simply connected harmonic manifold with $\rank(X) =1$. If $X$  admits a compact quotient,
then $X$ is a symmetric space of negative
curvature.
\end{cor}

\section{Gromov hyperbolic harmonic manifolds}

In this section we will show that for noncompact harmonic manifolds purely exponential volume growth is equivalent to Gromov hyperbolicity.
\begin{dfn} \label{G1}
Let $(X, d)$ be a metric space, $I \subset \mathbb{R}$ an Intervall. A curve $c: I \to X$
is called a geodesic, if $c$ is an isometry, i.e. $d(c(t), c(s))= |t-s| $ for $t, s \in I$.
A geodesic metric space $(X, d)$ is a  metric space where each pair of points can be joint by a geodesic.
\end{dfn}
\remark
Note that in Riemannian geometry geodesics are local isometries. Geodesics in the sense of metric spaces
correspond to minimal geodesics in Riemannian geometry.

There are several equivalent definition of Gromov hyperbolicity. The most common definition is the following.

\begin{dfn} \label{G2}
A geodesic metric space is called $\delta$-hyperbolic if all geodesic triangles are $\delta$-thin, i.e., each
side of a geodesic triangle is contained in the $\delta$-neighborhood of the two other sides.
A geodesic  metric space is called Gromov-hyperbolic if it is $\delta$-hyperbolic for some $\delta \ge 0$.
\end{dfn}

We want to show that Gromov hyperbolic harmonic manifolds have purely exponential volume growth. For that we will need the following elementary lemmata.
\begin{lemma}\label{G3}
Let $X$ be a simply connected manifold without conjugate points and $c_1,c_2: \mathbb{R} \to X$ be geodesics
with $c_1(0) = c_2(0)$ and $d(c_1(\pm \ell ,c_2( \mp \ell)) \le 1$ for $\ell > \frac{1}{2}$. Then
$$
d(c_1(t), c_2(s)) \ge 2 \ell -1 \quad \text{for all} \quad s, t \ge \ell.
$$
\end{lemma}
\begin{proof}
Consider $t, s \ge \ell$ and assume $s \le t$. Then
\begin{eqnarray*}
\ell + t &=& d(c_1( - \ell), c_1(t)) \le  d(c_1( - \ell), c_2( \ell)) \\
&& \quad +d(c_2( \ell), c_2(s)) +d(c_2(s), c_1(t)) \\
&\le& 1 + s- \ell +d(c_2(s), c_1(t))
\end{eqnarray*}
and therefore
$$
2 \ell -1 \le d(c_2(s), c_1(t))
$$
If $s \ge t$
\begin{eqnarray*}
\ell + s &=& d(c_2( - \ell), c_2(s)) \le  d(c_2( - \ell), c_1( \ell)) \\
&& \quad +d(c_1( \ell), c_1(t)) +d(c_1(t), c_1(s)) \\
&\le& 1 + t- \ell +d(c_1(t), c_2(s))
\end{eqnarray*}
and the assertion follows in this case as well.
\end{proof}

\begin{lemma}\label{G4}
Let $X$ be a simply connected $\delta$-hyperbolic manifold without conjugate points, $p \in X$
and $c:[0,a] \to X$ a geodesic. Let $c_1: [0, a_1] \to X$ be the geodesic joining $p$ and $c(0)$  and
 $c_2: [0, a_2] \to X$ be the geodesic joining $p$ and $c(a)$. Then there exist $t_1 \in [0, a_1]$,
 $t_2 \in [0, a_2]$ and $t_0 \in [0, a]$ such that
 $$
 d(c_1(t_1), c(t_0)) =  d(c_2(t_2), c(t_0)) \le \delta
 $$
\end{lemma}
\begin{proof}
Consider the continuous function $f:[0, a] \to \mathbb{R}$ given by
$$
f(t) = d(c(t), c_1[0,a_1]) - d(c(t), c_2[0,a_2])
$$
Therefore, $f(0) < 0$ and $f(a) > 0$ which implies the existence of $t_0 \in [0,a]$ with
$$
0 = f(t_0) =  d(c(t_0), c_1[0,a_1]) - d(c(t_0), c_2[0,a_2])
$$
Since by assumption geodesic triangles are $\delta$-thin we have
$$
d(c(t_0), c_1[0,a_1]) \cup  c_2[0,a_2]) \le \delta
$$
we obtain
$$
d(c(t_0), c_1[0,a_1]) = d(c(t_0), c_2[0,a_2]) \le \delta
$$
which implies the assertion of the lemma.
\end{proof}

\begin{cor}\label{G5}
Let $X$ be a simply connected $\delta$-hyperbolic manifold without conjugate points. Let
$c_1,c_2: \mathbb{R} \to X$ be geodesics with $c_1(0) =c_2(0) =p$ and
 $d(c_1(\pm \ell),c_2( \mp \ell)) \le 1$ , where $\ell := \delta +1$. For $a_1, a_2 > \ell$ consider
the geodesic $c:[0,a] \to X$ joining $c_1(a_1)$ and $c_2(a_2)$. Then there exists $t_0 \in [0,a]$
such that
$$
d(p, c(t_0)) \le 2 \delta +1
$$

\end{cor}
\begin{proof}
By lemma \ref{G4} there exist  $t_1 \in c_1[0, a_1]$,
 $t_2 \in c_2[0, a_2]$ and $t_0 \in c[0, a]$ such that
 $$
 d(c_1(t_1), c(t_0)) =  d(c_2(t_2), c(t_0)) \le \delta
 $$
 and therefore
 $ d(c_1(t_1), c_2(t_2))  \le 2\delta = 2 \ell -2$. Then
 $ \min(t_1, t_2) \le \delta +1 = \ell$ since otherwise lemma \ref{G3} would imply
 that $d( c_1(t_1) , c_2(t_2)) \ge 2 \ell -1$ which obviously is a contradiction.
Assume $0 \le t_1 \le \delta +1$ we obtain
$$
d(p, c(t_0)) \le d(p, c_1(t_1)) +d(c_1(t_1), c(t_0)) \le 2 \delta +1
$$
which yields the assertion.
\end{proof}

Let $X$ be simply connected manifold without conjugate points and $v \in S_pX$. Consider for $ t \ge 0$ the function
$ b_{v,t}(q) = d(q, c_v(t)) -t$. Then for all $q \in X$ the limit
$$
b_v(q) = \lim\limits_{t \to \infty}  b_{v,t}(q)
$$
exists
and defines the Busemann function $b_v$ associated to the geodesic $c_v$.
The levels of the Busemann functions are the horospheres.
It is easy to see  that $b_v$ is a $C^1$ function \cite{Es}  with $\|\grad b_v \| =1$ and
one can even prove \cite{Kn} that they are of class $C^{1,1}$, i.e., the $\grad b_v$ is Lipschitz.
This implies that the integral curves of $\grad b_v $ are geodesics
and $|b_v(q) - b_v(p)| \le d(p,q)$. In the case of
 simply connected noncompact harmonic manifold one can show  \cite{RSh3} that
 Busemann functions are analytic. Note that $\Delta b_v =h$, where $h$ is the mean curvature of the horospheres.
\begin{cor}\label{G6}
Let $X$ be a simply connected $\delta$-hyperbolic manifold without conjugate points. Consider
for $v \in S_pX$, $\ell = \delta +1$ and $r >0$ the spherical cone in $X$ given by
$$
A_{v, \ell}(r) := \{ c_w(t) \mid  0 \le t \le r , w \in S_pX, d(c_v( \pm \ell), c_w( \pm \ell)) \le 1 \}
$$
Then, for $\rho= 4 \delta +2$ the set $A_{v, \ell}(r) $ is contained in
\begin{equation*}
\begin{split}
H_{v,\rho}(r) :=\{ & c_q(t) \mid -\rho/2  \le t \le r , \; c_q \; \; \text{is an integral curve of} \; \; \grad b_{-v} \\
&\text{with} \; \; c_q(0) = q \in b_{-v}^{-1}(0) \cap B(p,\rho) \}
\end{split}
\end{equation*}
\end{cor}
\begin{proof}
For $c_w(t) \in A_{v, \ell}(r) $ there is a unique integral curve $c_q: \mathbb{R} \to X$ of $\grad b_{-v}$ such that $c_q(0) =q \in b_{-v}^{-1}(0)$
and $c_q(a) = c_w(t)$ for some $a \in \mathbb{R}$.
Let $c_{q ,s}$ be the sequence of geodesics with $c_{q ,s}(a) =c_q(a)$ and  $c_{q ,s}(b_s) = c_{-v}(s) $ for $b_s \le a$. Since
$$
d( c_{-v}(\mp \ell) ,  c_w( \pm \ell) =d( c_{v}(\pm \ell) ,  c_w( \pm \ell)) \le 1
$$
corollary \ref{G5} implies the existence of $x_s \in c_{q ,s}([b_s, a])$
such that $d(x_s, p) \le 2 \delta +1$. Note that
$
c_q = \lim\limits _{s \to \infty}c_{q ,s}.
$
Hence, there also exists $t_0 \le a$ such that
$d(c_q(t_0), p) \le 2 \delta +1$. Since $b_{-v}(c_q(t)) = t$ we obtain
$$
|t_0| = |b_{-v}(c_q(t_0))| \le |b_{-v}(c_q(t_0)) - b_{-v}(p)| \le d(c_q(t_0), p) \le 2 \delta +1
$$
and
$$
|a| = |b_{-v}(c_q(a))| \le |b_{-v}(c_w(t))) - b_{-v}(p)| \le d(c_w(t) , p) \le t.
$$
Therefore,
$$
d(p,q) = d(p, c_q(0)) \le d(p, c_q(t_0) )+ d(c_q(t_0), c_q(0)) \le 4 \delta +2.
$$
and  $-2\delta -1 \le t_0 \le a  \le t \le r$.
\end{proof}
Now we can prove the final result of this section.
\begin{theorem}\label{geocond}
Let $X$ be a simply connected noncompact harmonic manifold.
Then $X$ is $\delta$-hyperbolic for some $\delta >0$ if
the volume growth is purely exponential.
 Then $X$ has rank~1 and the geodesic flow is Anosov.
If $M$ is a compact harmonic manifold
with Gromov hyperbolic fundamental group, $M$ is a locally symmetric space of negative curvature.

\end{theorem}
\begin{proof}
Consider for $\ell = \delta +1$ and $v \in S_pX$ the set
$$
A_{v, \ell}(r) := \{ c_w(t) \mid  0 \le t \le r, \ w \in S_pX, \ d(c_v( \pm \ell), c_w( \pm \ell)) \le 1 \}
$$
Then
$$
\vol( A_{v, \ell}(r)) = \int\limits_0^r f(s) ds  \  \mu_p (C_ {v, \ell})
$$
where
$$
C_ {v, \ell}:= \{ w \in S_pX \mid  d(c_v( \pm \ell), c_w( \pm \ell)) \le 1  \}
$$
and  $\mu_p$ denotes the measure on the sphere $S_pX$ induced by Riemannian metric.
Corollary \ref{G6} implies that for $\rho= 4 \delta +2$ the set $A_{v, \ell}(r)$ is contained in
\begin{equation*}
\begin{split}
H_{v,\rho}(r) :=\{ & c_q(t) \mid -\rho/2  \le t \le r , \; c_q \; \; \text{is an integral curve of} \; \; \grad b_{-v} \\
&\text{with} \; \; c_q(0) = q \in b_{-v}^{-1}(0) \cap B(p,\rho) \}
\end{split}
\end{equation*}
Furthermore,
\begin{eqnarray*}
\vol(H_{v,\rho}(r) ) &=&\int\limits_{-\rho/2}^{r} e^{hs} ds \vol_{0}( b_v^{-1}(0) \cap B(p,\rho)) \\
&\le& \frac{e^{hr}}{h}\vol_{0}( b_v^{-1}(0) \cap B(p,\rho)),
\end{eqnarray*}
where $\vol_{0}$ denotes the induced volume on the horosphere $b_v^{-1}(0)$.
Therefore,
$$
\int\limits_0^r f(s) ds  \  \mu_p (C_ {v, \ell})
\le  \frac{e^{hr}}{h}\vol_{0}( b_v^{-1}(0) \cap B(p,\rho))
$$
and the ratio
 $$
 \frac{\int\limits_0^r f(s) ds}{e^{hr}}
 $$
 is bounded above by a constant. Therefore, the ratio
 $$
 \frac{ f(r) }{e^{hr}}
 $$
 is bounded from above as well and
 by corollary~\ref{volgr} this implies that $X$ purely exponential volume growth.\\
 \end{proof}
 Now we are able to prove the following main result on simply connected noncompact harmonic manifolds stated in the introduction
 \begin{theorem}\label{mainthm1}
Let $X$ be a simply connected noncompact harmonic manifold. Then the following assertion are equivalent:
\begin{enumerate}
\item[(i)]
$X$ is Gromov hyperbolic.
\item[(ii)]
 $X$ has purely exponential volume growth.
 \item[(iii)]
  $X$ has rank one.
  \item[(iv)] $X$ has an Anosov geodesic flow with respect to the Sasaki metric
\end{enumerate}
\end{theorem}
\begin{proof}
  (i) implies (ii) by the previous theorem.
  The equivalence of (ii), (iii), and $(iv)$ has been obtained in the previous section.\\
  Assume now that the geodesic flow $\phi^t:SX \to SX$ is Anosov
  with respect to the Sasaki metric. For $v \in SX$ consider the Jacobi tensor with
   $A_v(0) = 0$ and $A_v'(0) = \id$. Then the Anosov condition implies
(see \cite{Bo} )
$$
\|A_v(t)x \|\ge \|x\| e^{\alpha t}
$$
Consider two distinct geodesic rays $c_1:[0, \infty) \to X$ and 
$c_2:[0, \infty) \to X$ with $c_1(0) =c_2(0) = q$ and define
\begin{eqnarray*}
d^q_t(c_1(t), c_2(t)) &:= &\inf \{ L(\gamma) \mid  \;
\gamma:[a,b] \to  X \setminus B(q,t) \\
&& \text{ a piecewise smooth curve
joining } \; c_1(t) \text{ and } c_2(t) \}.
\end{eqnarray*}
Then 
$$
\liminf_{t \to \infty}\frac{d^q_t(c_1(t), c_2(t))}{t} \ge \alpha
$$
Using proposition 1.26 in chapter III of \cite{BH} this implies that
$X$ is 
  Gromov hyperbolic.
  \end{proof}
  The following theorem was known in the case of negative curvature
 \begin{theorem}
 Let $M$ be a compact harmonic manifold
with Gromov hyperbolic fundamental group.
 Then $M$ is a locally symmetric space of negative sectional curvature.
 \end{theorem}
\begin{proof}
Since the fundamental group is Gromov hyperbolic the universal cove is
Gromov hyperbolic as well. Hence, the
geodesic flow is Anosov and therefore, $M$ is a locally symmetric space of negative curvature.
 \end{proof}

We strongly believe that noncompact harmonic, nonflat manifolds with $h>0$ and higher rank do not exists.
The main purpose of the next section is to prove this under the assumption of no focal points.\\
\section{Harmonic manifolds with bounded asymptote}

In this section, we show now that for a harmonic manifold of bounded asymptote the rank is constant, i.e., independent of the geodesic.
In odd dimensions this implies by a result of N.E.~Steenrod and J.H.C. Whitehead \cite{SW} that the rank is 1.
\begin{dfn}
Let $M$ be a manifold without conjugate points and $\alpha \ge 1$. A geodesic $c_v: \mathbb{R} \to M$ is called $\alpha$-stable if
$$
\|S_v(t)x_t \| \le  \alpha\| x \|      \; \; \text{and} \; \;     \|U_v(t)x_t \| \ge \frac{1}{\alpha} \| x \|,
$$
for all $t \ge 0$ and parallel vector fields $x_t$  with $x_0 = x \in  v^\perp$. We call a geodesic stable if it is $\alpha$-stable for some constant
$\alpha \ge 1$.\\
$M$ is called of bounded asymptote if there is uniform constant $\alpha \ge 1$ for which all geodesics are
$\alpha$-stable.

\end{dfn}
\begin {remark}
The notion of bounded asymptote has been introduced by J.-H.~Eschenburg \cite{Es}. In particular, if $M$ has nonpositive curvature or, more generally, no focal points, each geodesic is $1$-stable.
\end {remark}
\begin{prop} \label{constrank}
Let $X$ be a noncompact simply connected harmonic manifold. Suppose there exists a $\alpha(v)$-stable geodesic $c_v$ with $\rank(v) =k+1 \ge 2$. Then for all $x \in \ker( U_{v}'(0) - S_{v}'(0))$
$$
  \frac{1}{\alpha^2(v) t} \langle x,  x \rangle  \le \langle (U_{v}'(0) - S_{v,t}'(0))x, x \rangle \le \frac{\alpha^2(v)}{t} \langle x,  x \rangle.
 $$
 Let
 $$
 \lambda_1(v,t) \le \ldots \lambda_k(v,t) \le \lambda_{k+1}(v,t) \le \ldots \le \lambda_{n-1}(v,t)
 $$
 be the eigenvalues of $ U_{v}'(0) - S_{v,t}'(0)$ and
 $\beta_t(v) :=  \lambda_{k+1}(v,t) \cdot \ldots \cdot  \lambda_{n-1}(v,t)$, then $\beta_t(v)$ is monotonically decreasing and converging to the product of the positive eigenvalues  $\beta(v)$ of $ U_{v}'(0) - S_{v}'(0)$. Furthermore,
 $$
 \frac{\beta(v)}{\alpha^{2k}(v)} \le \det (U_{v}'(0) - S_{v,t}'(0)) t^k\le  \alpha^{2k}(v) \beta_t(v)
 $$
 \end{prop}

\begin{proof}
Let $\alpha(v) \ge 1$ and  $c_v$ is an $\alpha(v)$-stable geodesic.
Since lemma~\ref{sujac} implies
$$
(U_{v}'(0) - S_{v,t}'(0)) = \left(\int\limits_{0}^t \left(U_{v}^{\ast}U_{v} \right)^{-1}(u) du \right)^{-1}
$$
we have for each unit vector $x \in v^\perp$
\begin{eqnarray*}
\langle (U_{v}'(0) - S_{v,t}'(0))x, x \rangle
&\ge &\frac{1}{ \max \left\{ \int\limits_{0}^t \langle(U_{v}^{\ast}U_{v})^{-1}(u)y_u, y_u \rangle du \mid \| y\| =1  \right\} }\\
&=&\frac{1}{ \left\| \int\limits_{0}^t (U_{v}^{\ast}U_{v})^{-1}(u) du \right\|}
\end{eqnarray*}
Using
$$
 \left\| \int\limits_{0}^t (U_{v}^{\ast}U_{v})^{-1}(u) du \right\| \le \int\limits_{0}^t \|(U_{v}^{\ast}U_{v})^{-1}(u)\| du  \le  \int\limits_{0}^t \| U_{v}^{-1}(u)\|^2 du
 $$
 and
 $$
 \|U_{v}^{-1}(u)\| = \frac{1}{ \min \{ \|U_{v}(u)x_u\| \mid  \| x\| =1 \} } \le \alpha(v),
 $$
 we obtain
$$
 \frac{1}{\alpha^2(v) t} \le \langle (U_{v}'(0) - S_{v,t}'(0))x, x \rangle .
 $$
 On the other hand, lemma~\ref{sujac} implies
$$
(S_{v}'(0) - S_{v,t}'(0)) = \left(\int\limits_{0}^t(S_{v}^{\ast}S_{v})^{-1}(u) du \right)^{-1}
$$
as well.
Using the estimate \eqref{symtensor}, we obtain for all unit vectors $x \in v^\perp$
\begin{equation*}
\begin{split}
\langle (S_{v}'(0) - & S_{v,t}'(0))x, x \rangle \le \left\| \left(\int\limits_{0}^t(S_{v}^{\ast}S_{v})^{-1}(u) du \right)^{-1} \right\| \\
&\le   \left(\int\limits_{0}^t \| (S_{v}^{\ast}S_{v})(u) \| ^{-1}du \right)^{-1}
= \left(\int\limits_{0}^t \|S_{v}(u) \| ^{-2}du \right)^{-1}\\
& \le  \left(\int\limits_{0}^t \frac{1}{\alpha^2(v)}du \right)^{-1}= \frac{\alpha^2(v)}{t}.
\end{split}
\end{equation*}
 Therefore,
 $$
 \langle (S_{v}'(0) - S_{v,t}'(0))x, x \rangle \le \frac{\alpha^2(v)}{t}.
 $$
 Putting both inequalities together, we obtain
 for all $x \in \ker(U_{v}'(0) - S_{v}'(0))$ with $\|x \| =1$:
 $$
  \frac{1}{\alpha^2(v) t} \le \langle (U_{v}'(0) - S_{v,t}'(0))x, x \rangle \le \frac{\alpha^2(v)}{t},
 $$
 which implies the first assertion.
Let $0 < \lambda_1(v,t) \le \ldots \le \lambda_{n-1}(v,t)$ the eigenvalues of $(U_{v}'(0) - S_{v,t}'(0))$ and
$k = \dim( \ker(U_{v}'(0) - S_{v}'(0))$. Then using the above estimate
$$
 \frac{1}{\alpha^2(v) t} \le  \lambda_i(v,t) \le \frac{\alpha^2(v)}{t}
 $$
 for $1 \le i \le k$. The remaining eigenvalues $\lambda_{k+1}(v,t) \le \ldots \le \lambda_{n-1}(v,t)$ of $(U_{v}'(0) - S_{v,t}'(0))$ are monotonically decreasing in $t$ and  converging to the positive eigenvalues
 of $(U_{v}'(0) - S_{v}'(0))$. Hence,
 \begin{equation}\label{4B1}
 \frac{ \beta(v) }{\alpha^{2k}(v) t^k} \le \det (U_{v}'(0) - S_{v,t}'(0)) \le \frac{\alpha^{2k}(v) \beta_t(v) } {t^k},
 \end{equation}
 where $\beta_t(v) = \lambda_{k+1}(v,t) \cdot \ldots \cdot \lambda_{n-1}(v,t)$  and $\beta(v) =  \lambda_{k+1}(v) \cdot \ldots \cdot \lambda_{n-1}(v)$ is the product of the positive eigenvalues.\\
  \end{proof}
\begin{cor}
Let $X$ be a noncompact and nonflat harmonic mani\-fold  having a stable geodesic
 $c_v$. Then there is a constant
$  b \ge 1$ such that
$$
\frac{1}{b} \le \frac{f(t)}{e^{ht} \ t^{\rank(v) -1} } \le b
$$
for all $t \ge1$.
\end{cor}
\begin{proof}
Using corollary~\ref{volgr}, we have
$$
\frac{e^{ht}}{f(t)} = \det (U_v'(0) - S_{v,t}'(0)).
$$
If $k = \rank(v)-1 $ the estimate follows from proposition \ref{constrank}.
\end{proof}
  \begin{cor}\label{4D}
  Let X be a noncompact simply connected harmonic mani\-fold.
  Then $\rank(v)$  is constant on the set of initial conditions $v\in  SM$
 corresponding to stable geodesics. For a fixed $\gamma  \ge 1$ consider the set
 $$
 G_\gamma := \{ v \in SM \mid c_v  \; \text{is an} \; \alpha(v)-\text{stable geodesic with } \alpha(v) \le \gamma \}.
 $$
 Then, there exists a constant $\rho >0$  such that
  $$
\langle (U_v'(0) - S_v'(0))x, x \rangle \ge \rho \langle x, x \rangle
$$
 for alle $x \in  \ker (U_{v}'(0) - S_{v}'(0))^{\perp} \subset v^\perp $ and $v \in G_\gamma$.
  \end{cor}
 \begin{proof}
  Since the function $\det (U_{v}'(0) - S_{v,t}'(0))$ is independent of $v \in SX$  proposition \ref{constrank}
 implies that $\rank(v)$ is constant for all $v$ corresponding to stable geodesics.
 Now let $v_0$ be a fixed and $v$ an arbitrary vector in $G_\gamma$. Then
 the estimate \eqref{4B1} implies
\begin{eqnarray*}
 \frac{ \beta(v_0) }{\alpha^{2k}(v_0)} &\le& \det (U_{v_0}'(0) - S_{v_0,t}'(0))t^k  = \det (U_{v}'(0) - S_{v,t}'(0))t^k\\
 &\le& \alpha(v)^{2k}\beta_t(v) \le  \gamma^{2k} \beta_t(v)
 \end{eqnarray*}
 for all $t\ge 0$ and therefore,
 $$
 \frac{ \beta(v_0) }{\alpha^{2k}(v_0)} \le  \gamma^{2k} \beta(v)
 $$
where  $\beta(v) =\prod\limits_{i= k+1}^{n-1} \lambda_i(v)$  is the product of the positive eigenvalues of $(U_{v}'(0) - S_{v}'(0))$.
 Since the curvature of $M$ is bounded the eigenvalues of $(U_v'(0) - S_{v}'(0))$
are bounded by lemma~\ref{3.1A} from above. Since $\beta(v) $ is bounded from below where the bound only depends on  $\gamma$, $\alpha(v_0)$ and $\beta(v_0)$, the same is true for all the positive eigenvalues of $(U_v'(0) - S_{v}'(0))$.
\end{proof}
\begin{theorem}
Let $X$ be a nonflat noncompact harmonic manifold of odd dimension.
Suppose there exists $p \in X$ such that all geodesics $c_v$ with initial conditions $v \in S_pX$ are stable. Then $X$ has rank~1.
\end{theorem}
\begin{proof}
Since the $\rank(v)$ is constant for each $v \in S_pX$ consider the subspace  of $v^\perp \cong T_vS_pX$, given by
$
\cal L(v) = \ker (U_{v}'(0) - S_{v}'(0)).
$
Since the rank is constant it defines a continuous distribution. If the dimension of $X$ is odd, the dimension of the sphere $S_pX$ is even. By a result of N.E.~Steenrod and J.H.C.~Whitehead \cite{SW}
 the distribution must be trivial, i.e., $\dim \cal L(v)$ is zero or has dimension $n-1$. In
the latter case $U_{v}'(0) =S_{v}'(0)$ and, therefore, $e^{ht} =\det U_v(t) = \det S_v(t)  = e^{-ht}$. But then
$h =0$ and $X$ is flat. Hence, $\dim \cal L(v)$ is zero and $\rank( X) =1$.
\end{proof}

\section{Harmonic manifolds without focal points}

In this section, we will show that a nonflat simply connected harmonic manifold $X$  without focal points has rank~1.
For odd dimensions the prove has been given in the previous section.
Note that $X$ has no focal points if the second fundamental form of horospheres is positive semi-definit
or if for all $v \in SX $ the Busemann functions $b_v$ are convex.
In terms of the stable and unstable Jacobi tensor $S_v$ and $U_v$ this is equivalent to
$U_v'(0) \ge 0$ and therefore $S_v'(0) = -  U_v'(0)\le 0$ for all $v \in SX$.
Most of the properties of manifolds of nonpositive
curvature are shared by manifolds without focal points. For instance geodesic spheres are convex
since  horospheres are convex. Furthermore, the flat strip theorem is true, which asserts that two geodesics $c_1$ and $c_2$ bound a flat strip
if $d(c_1(t) , c_2(t)) \le b$ for some $b \ge 0$ and all $t \in \mathbb{R}$. This means that there exists an isometric, totally geodesic imbedding $F: [0,a] \times \mathbb{R} \to X$ such that $c_1(t) = F(0,t)$ and $c_2(t) = F(a,t)$ (see \cite{Es} for a proof).
In particular, $c_1 $ and $c_2$ are parallel, i.e., $d(c_1(t), c_2(t))=a$.

\begin{lemma} \label{5A}
Let $X$ be a manifold without focal points.
Then the following holds
\begin{eqnarray*}
\cal L(v) &=& \ker (U_{v}'(0) )\cap   \ker (S_{v}'(0))   \subset v^\perp \\
 &=& \{x \in v^\perp \mid x_t \; \text{is a parallel Jacobi field along } \; c_v\}.
\end{eqnarray*}
Furthermore $x_t$,$ S_v(t) x_t $, $U_v(t) x_t \in  \cal L^\perp(\phi^tv) $ for all
 $x \in \cal L^\perp(v) =\{x \in v^\perp \mid x \perp \cal L(v) \}$.
 \end{lemma}

\begin{proof}
 Since $X$ has no focal points, we have
$$\langle (S_{v}'(0)x, x \rangle \le 0 \le \langle (U_{v}'(0)x, x \rangle
 $$
 for all $x \in v^\perp$. Since $S_{v}'(0)$ and  $U_{v}'(0)$ are symmetric endomorphisms
$x \in  \ker (U_{v}'(0) - S_{v}'(0))$ implies $x \in \ker (U_{v}'(0) \cap   \ker (S_{v}'(0)) $.\\
 Assume that $x \in \cal L(v)$. Then lemma~\ref{central} implies $U_v(t) x_t \in \cal L(\phi^tv)$ and we obtain
$$
(U_v(t) x_t)' = U_v'(t) x_t = U_v'(t) U_v^{-1}(t) U_v(t)x_t =   U_{\phi^tv}'(0) U_v(t)x_t = 0.
$$
Therefore, $U_v(t) x_t$ is parallel along $c_v$. In particular, $U_v(0) = \id $ yields $U_v(t) x_t = x_t$.\\
Now consider $x \in v^{\perp}$ such that its parallel translation defines a Jacobi field. Then for each $s \not=0$
$J_s(t) = \frac{s-t}{s}x_t$ defines a Jacobi field with $J_s(0) =x$ and $J_s(s) =0$. Therefore,
$$
U_{v}x_t = \lim\limits_{s \to \infty}J_s(t) = x_t =  \lim\limits_{s \to -\infty} J_s(t)= S_{v}x_t.
$$
Assume that
$x \in  \cal L^\perp(v)$. Then for all $y \in \cal L(v)$, we have $\langle y_t, x_t \rangle = \langle x,y \rangle$ and $y_t \in \cal L(\phi^tv)$ implies $ x_t \in \cal L^\perp(\phi^t(v) )$.\\

To prove the last assertion consider $x \in \cal L^\perp(v)$. Then for all $y \in \cal L(v)$, we obtain
\begin{eqnarray*}
\langle U_v(t) x_t , y_t \rangle '&= &\langle U_v'(t) x_t , y_t \rangle =
\langle U_v'(t)  U_v^{-1} (t) U_v(t) x_t , y_t \rangle\\
& =& \langle U_{\phi^tv}'(0)  U_v(t) x_t , y_t \rangle = \langle  U_v(t)x_t,  U_{\phi^tv}'(0)y_t \rangle =0
\end{eqnarray*}
Hence, $\langle U_v(t) x_t , y_t \rangle = \langle x, y \rangle = 0$ and, therefore,
$U_v(t) x_t \in \cal L^\perp(\phi^t(v))$. In the same way one proves: $S_v(t) x_t \in \cal L^\perp(\phi^t(v))$.
\end{proof}
\begin{lemma} \label{5A1}
Let $X$ be a simply connected manifold without focal points such that $\dim \cal L(v)$ is positive and independent of $v$.
Then, for each compact set $K \subset SX$ there exists $T >0$ such that
 $$
 \cal L(v) = \{ x \in v^\perp \mid R(x_t, \dot{c}_v(t), \dot{c}_v(t)) =0 , t \in [-T, T]\}.
  $$
 for all $v \in K$. In particular
 $\cal L(v)$ and $\cal L^\perp(v)$ depend smoothly on $v \in SX$.
 \end{lemma}
\begin{proof}
According to Lemma \ref {5A} we have that $x \in \cal L(v)$ if and only if its parallel translation $x_t$ is a  Jacobi field
along $c_v$. Therefore, each $x \in \cal L(v)$ is contained in
$$
\cal L_T(v) :=\{ x \in v^\perp \mid R(x_t, \dot{c}_v(t), \dot{c}_v(t)) = 0, t \in [-T, T]\}
$$
 for all $T >0$.
Let $K \subset SX$ be compact. If $T$ with the required property would not exists, we could choose
a convergent sequence $v_n \in K$ and a sequence $T_n$ with $T_n \to \infty$ such
that for $v = \displaystyle\lim\limits_{n \to \infty} v_n =v$
$$
\dim \cal L_{T_n}(v_n) > \dim \cal L(v_n) = \dim \cal L(v)
$$
which would imply
$$
\dim  \{ x \in v^\perp \mid R(x_t, \dot{c}_v(t), \dot{c}_v(t)) =0, t \in \mathbb{R} \} > \dim \cal L(v).
$$
On the other hand each $x \in v^\perp$ with
$R(x_t,  \dot{c}_v(t), \dot{c}_v(t)) = 0$ for all $t \in \mathbb{R}$ defines a parallel Jacobi field
along $c_v$ and hence is contained in $\cal L(v)$ which leads to a contradiction.
Hence, locally  $\cal L(v) = \cal L_T(v)$ for sufficiently large $T$. In particular, $\cal L_T(v)$ and   $\cal L^\perp(v)$ depend smoothly on $v \in SX$.
\end{proof}

\begin{prop}\label{5B}
Let $X$ be a nonflat manifold without focal points
and bounded sectional curvature. Assume that the eigenvalues of
$(U_{v}'(0) - S_{v}'(0))$ restricted to $\cal L^\perp (v)\subset v^\perp$ are bounded from below by a positive constant independent of $v$. Then
there are constants $a \ge 1$ and $\alpha >0$ independent of $v$ such that
$$
\| S_v(t)x_t \| \le ae^{-\alpha t} \|x \| \; \; \text{and} \; \; \| S_v(-t)x_{-t} \| \ge \frac{1}{a} e^{\alpha t} \|x \|
$$
as well as
$$
\| U_v(t)x_t \| \ge \frac{1}{a} e^{\alpha t} \|x \| \; \; \text{and} \; \;\| U_v(-t)x_{-t} \| \ge ae^{-\alpha t} \|x \|
$$
for all $x \in \cal L^\perp(v) $ and $t\ge 0$.
Furthermore, for $t \ge 0$
$$
\| A_v(t)x_t \| = t\|x\|  \quad \text{for all } \; x \in \cal L (v)
$$
and  there exists a constant $a'$ such that
$$
\| A_v(t)x_t \| \ge a' e^{\alpha t}\|x\| \quad \text{for all} \;x \in \cal L^\perp(v)
$$
for all $t \ge 1$.
\end{prop}
\begin{proof}
Consider the equation
$$
S_v^{\ast}(t) \left(U_{\phi^t(v)}'(0) - S _{\phi^t(v)}'(0) \right) S_v^{\ast}(t)=  \left(\int\limits_{-\infty}^{t} (S_v^{\ast}S_v)^{-1}(u) du \right)^{-1}
$$
proved in lemma~\ref{central} of the appendix. Using Lemma \ref{5A} the left and therefore
also the right hand side defines a strictly positive symmetric endomorphism on $\cal L^\perp(\phi^t v)$.
Since the eigenvalues of $(U_{v}'(0) - S_{v}'(0))$ restricted to $ \cal L^\perp (v)\subset v^\perp$
are bounded from below by a positive constant $\rho$, lemma~\ref{central} of the appendix implies
\begin{eqnarray*}
 \rho &\le& \left \langle \left( \int\limits_{-\infty}^{t} (S_v^{\ast}S_v)^{-1}(u) du
 \right)^{-1}_{\Big |\cal L^\perp(\phi^t v)}  S_v^{-1}(t)x, S_v^{-1}(t)x \right \rangle\\
&\le& \left\| \left(\int\limits_{-\infty}^{t} (S_v^{\ast}S_v)^{-1}(u) du \right)^{-1}_{\Big |\cal L^\perp(\phi^t v)}  \right\| \left\| S_v^{-1}(t)x \right\|^2
\end{eqnarray*}
for all $x \in \cal L^\perp(\phi^t v)$ with $\|x\| =1$.
Furthermore,  we have
\begin{equation*}
\begin{split}
&\left\| \left(\int\limits_{-\infty}^{t} (S_{v}^{\ast}S_{v})^{-1}(u) du \right)^{-1}_{\Big|\cal L^\perp(\phi^t v)} \right\| =\\
& \quad \quad \frac{1} { \min \left
\{\int\limits_{-\infty}^{t} \langle (S_{v}^{\ast}S_{v})^{-1}(u)y_u, y_u \rangle du \mid y \in \cal L^\perp(v), \|y\| =1 \right\} }.
\end{split}
\end{equation*}
Therefore,
\begin{equation*}
\begin{split}
\rho \min &\left\{\int\limits_{-\infty}^{t} \langle S_v^{\ast -1}(u)y_u, S_v^{\ast -1}(u)y_u \rangle du \mid y \in \cal L^\perp(v), \|y\| =1) \right\}\\
& \quad \quad \le \|S_v^{-1}(t)x \|^2.
\end{split}
\end{equation*}
Defining
\begin{eqnarray*}
\varphi(u) &:= &\min \left\{\|S_v^{\ast -1}(u)y \|^2 \mid y \in \cal L^\perp(\phi^u v), \|y\| = 1 \right\}\\
&=& \min \left\{\|S_v^{ -1}(u)y \|^2 \mid y \in \cal L^\perp(\phi^u v), \|y\| = 1 \right\},
\end{eqnarray*}
 we obtain
$$
\rho \int\limits_{0}^{t} \varphi(u)du \le \rho \int\limits_{-\infty}^{t} \varphi(u)du \le \varphi(t)
$$
and, hence,
$$
\rho F(t) \le F'(t)
$$
for $F(t) := \int\limits_{0}^{t} \varphi(u)du $. This implies $F(t) \ge F(1) e^{\rho t}$ for all $t \ge 1$ and, therefore,
$\varphi(t)= F'(t) \ge \rho F(t)\ge \rho F(1) e^{\rho t}$ for all $t \ge 1$.
 Since the sectional curvature of $X$ is bounded, $\varphi(t)$ is on $[0,1]$ bounded away from $0$.
 Hence, there exists some constant $a \ge 1$ such that
\begin{equation} \label{eq5a}
\|S_v^{-1}(t)y \| \ge \frac{1}{a} e^{\alpha t}\|y\|
\end{equation}
for $ \alpha = \frac{\rho}{2}$ and all $y \in \cal L^\perp (\phi^tv) $ and $t \ge 0$. Since  $S_v(t) :  \cal L^\perp(\phi^tv) \to  \cal L^\perp(\phi^tv)$
 is an isomorphism, we obtain for all $x \in \cal L^\perp(\phi^tv)$ and $t \ge 0$
$$
\|S_v(t)x \| \le a e^{-\alpha t} \|x \|.
$$
Note, that for each $u \in \mathbb{R}$, we have
$$
S_{\phi^u v}(t) x_t = S_v(t+u)(S_v^{-1}(u)x)_t,
$$
where $x_t$ is the parallel translation of $x \in \cal L^\perp(\phi^u(v))$ along $c_{\phi^u v}(t)$. Hence,
for $t = -u \le 0$ we obtain with \eqref{eq5a}
$$
\| S_{\phi^u v}(-u) x_{-u} \| = \| (S_v^{-1}(u)x)_{-u} \| = \| (S_v^{-1}(u)x) \| \ge
\frac{1}{a} e^{\alpha u} \|x \|.
$$
In particular, for $w = \phi^u v$ and $u \ge 0$
the estimate
$$
\| S_w(-u) x_{-u} \| \ge \frac{1}{a} e^{\alpha u} \|x \|
$$
holds for all $x \in \cal L^\perp (w)$. Since
$ U_v(t) = S_{-v}(-t)$ the second estimate of the proposition follows.\\
To prove the remaining assertions we recall that
$$
A_v(t) x_t = U_v(t) \int\limits_{0}^{t} \left(U_v^{\ast}U_v \right)^{-1}(s)x_s ds.
$$
If $x \in   \cal L (v)$ we have $\left(U_v^{\ast}U_v \right)^{-1}(s)x_s =x_s$ and, therefore,
$A_v(t) x_t = U_v(t) (tx_t) = t x_t$. \\
If $x \in   \cal L^\perp (v)$ we have $\int\limits_{0}^{t} \left(U_v^{\ast}U_v \right)^{-1}(s)x_s ds \in   \cal L^\perp (\phi^tv)$. Therefore,
$$
\| A_v(t) x_t \| \ge a e^{\alpha t} \left\| \int\limits_{0}^{t} \left(U_v^{\ast}U_v \right)^{-1}(s)x_s ds  \right\|.
$$
Since lemma~\ref{sujac} implies
$$
((U_{v}'(0) - S_{v,t}'(0))^{-1}x)_t = \int\limits_{0}^t(U_{v}^{\ast}U_{v})^{-1}(s)x_s ds
$$
and since for $t \ge 1$ there exists a constant $b >0$ such that
$\langle (U_{v}'(0) - S_{v,t}'(0))x ,x \rangle \le b \langle x, x \rangle$
for all $x \in v^\perp$ the last estimate follows.
\end{proof}

\begin{theorem}\label{5C}
Let X be a nonflat
manifold without focal points and bounded sectional curvature. Assume that the rank of X is constant, i.e.,
the dimension of
$$
\dim \cal L(v) = \dim \{ x \in v^\perp \mid v \in \ker (U_v'(0))\cap   \ker (S_v'(0)) = \rank(v) -1
$$
is independent of
$v$. Assume furthermore, that there exists a constant $\rho >0$ such that
$$
\langle (U_{v}'(0) - S_{v}'(0))x , x \rangle \ge \rho \langle x, x \rangle
$$
for all $x \in \cal L^\perp(v)$. Then the rank is equal to one and the geodesic flow is Anosov.
\end{theorem}
As a consequence we obtain:
\begin{theorem}
 Let  $X$ be a simply connected nonflat harmonic manifold without focal points. Then $X$ is of rank~1. Moreover, the geodesic flow is Anosov and $X$ is Gromov hyperbolic.
\end{theorem}
\begin{proof}
Since $X$ is harmonic and has no focal points corollary~\ref{4D} implies that the
condition of theorem~\ref{5C} are fulfilled.
\end{proof}
{\bf It remains to prove \ref{5C}.}\\
We first show that the geodesic flow is partially hyperbolic.
Consider the distributions
\begin{eqnarray*}
E^p(v) &=&\{(x + \lambda v ,0)  \mid x \in \cal L(v) , \lambda \in \mathbb{R} \},\\
E^c(v) &=&\{(x + \lambda v ,y)  \mid x, y \in  \cal L(v), \lambda \in \mathbb{R} \},\\
E^s(v)&= &\{x, S_v'(0)x)  \mid x \in \cal L^\perp(v)  \},\\
E^u(v) &= &\{x, U_v'(0)x)  \mid x \in \cal L^\perp(v) \}.
\end{eqnarray*}
We call $E^p$ the parallel,  $E^c$ the central, $E^s$ the stable and $E^u$ the unstable distribution.
Obviously $E^p(v) \subset E^c(v)$. Furthermore, the central, stable and the unstable distribution
are transversal and
$$
\dim E^c (v) = 2  \rank(X) -1, \; \dim E^s (v) = \dim E^u (v) = n - \rank X.
$$
Therefore, the sum of the dimension is equal to $2n-1$ and
hence,
$$
T_vSX = E^c (v) \oplus E^s (v) \oplus E^u(v).
$$
\begin{lemma}
Under the assumption of theorem~\ref{5C}, the geodesic flow is partially hyperbolic with respect to the
Sasaki metric. More precisely, there are constants $ b, c \ge 1$ such that for all $\xi \in E^s(v)$
$$
\| D\phi^t(v) \xi \| \le b \| \xi \| e^{\alpha t}, \; t \ge 0 \;\;  \text{and} \;\;  \| D\phi^t(v) \xi\| \ge \frac{1}{b} \| \xi \| e^{- \alpha t}, \; t \le  0
$$
and for all $\xi \in E^u(v)$
$$
\| D\phi^t(v) \xi \|\ge \frac{1}{b} \| \xi \| e^{\alpha t}, \; t \ge 0 \; \;\text{and} \; \| D\phi^t(v) \xi \| \le b \| \xi \| e^{\alpha t}, \; \; t \le  0.
$$
Furthermore, for all $\xi \in E^c(v)$ and $t \in \mathbb{R}$ we have
$$
\| D\phi^t(v) \xi \| \le c \| \xi \| ( |t| +1),
$$
\end{lemma}
\begin{proof}
For  $\xi = (x, S_v'(0)x) \in E^s(v)$ we obtain
\begin{eqnarray*}
\| D\phi^t(v) \xi \|& =& \| (S_v(t)x_t , S_v'(t)x_t) \| = \sqrt{\| S_v(t)x_t \|^2 + \| S_v'(t)x_t)\|^2 }\\
&=& \| S_v(t) x_t \| \sqrt{ 1 + \frac{\| S_v'(t)x_t)\|^2}{\| S_v(t)x_t \|^2}} .
\end{eqnarray*}
Moreover, lemma \ref{3.1A} implies
$$
\frac{\| S_v'(t)x_t)\|}{\| S_v(t)x_t \|} \le   \| S_v'(t) S_v(t)^{-1} \| = \| S_{\phi^tv}'(0)\| \le \beta
$$
and proposition~\ref{5B} yields
$$
 \| D\phi^t(v) \xi \| \le  a e^{-\alpha t}\|x\| \sqrt{ 1 + \beta^2} \le  a e^{-\alpha t}\|\xi\| \sqrt{ 1 + \beta^2}
$$
The remaining assertion are obtained in a similar way.
\end{proof}

As we will see all the distributions are integrable. Define
$$
P(v) = \{ w \in SX \mid d(c_w(t) , c_v(t)) \; \text{is constant} \}
$$
to be the subset of $SX$ consisting of vectors tangent to the parallel geodesics of $c_v$
and denote by $F(v) := \pi(P(v))$ its projection on $X$.

\begin{prop}
Assume that $X$ fulfills the assumption of theorem ~\ref{5C}.
Then the distributions $E^p$ and $E^c$ are integrable and provide flow invariant and smooth foliations. The integral manifolds of $ E^p$ are given by $P(v)$.
 The projection $F(v) := \pi(P(v))$ of each leave is a $k$-flat. The integral manifolds of $ E^c$
are given by the unit tangent bundles $SF(v)$ of the $k$-flats.
The distributions $E^s$ and $E^u$ are integrable as well and the leaves are the stable and unstable manifolds
given by
$$
W^s(v) = \{ w \in SX \mid d(c_v(t) , c_w(t)) \le a e^{-c t} d(\pi(v) , \pi(w)) , t \ge 0 \}
$$
and
$$
W^u(v) = \{ w \in SX \mid d(c_v(t) , c_w(t)) \le a e^{-c t} d(\pi(v) , \pi(w)) , t \le 0 \}
$$
for constants $c,a >0$.
\end{prop}
\begin{proof}
The integrability of $E^p$ follows as in  lemma~2.2 in \cite{BBE}, which was given under the assumption of nonpositive curvature and compact quotient, but not necessarily constant rank. By lemma \ref{5A1} the
 distribution $E^p(v) = \{ (x + \lambda v, 0) \mid x \in \cal L(v)\}$ is smooth. Choose a smooth curve
$\rho: [0, a] \to SX $ with $\rho(0) =v$  tangent to $E^p$, i.e. $\frac{d}{ds} \rho(s) =  (x(s) + \lambda(s) \rho(s), 0) \in E^p(\rho(s))$, and therefore, $x(s) \in \cal L(\gamma(s))$. Consider for each $t \in \mathbb{R}$ the curve
$\gamma_t:I \to SX$ with $\gamma_t(s) = c_{\rho(s)}(t)$. Hence,
$$
 \frac{d}{ds}\gamma_t(s)= J_{\rho(s)}(t),
 $$
 where $J_{\rho(s)}(t)$ is the parallel Jacobi field with $J_{\rho(s)}(0)=x(s) + \lambda(s) \rho(s)$.
Consequently,  the length of $\gamma_t(I )$ is constant, the distance $d(c_{v}(t),c_{\rho(s)}(t)$
is bounded and the geodesic $c_{\rho(s)}(t)$ is parallel to $c_v$. If $\xi$ and $\eta$ are two smooth vector
fields tangent to $E^p$ the commutator $[\xi, \eta]$ is tangent to $E^p$ as well.
To prove this consider the flows $\varphi_\xi$ and $\varphi_\eta$. Then for $s$
$$
f(s^2) =  \varphi_\eta^{-s} \circ \varphi_\xi^{-s} \circ \varphi_\eta^s \circ \varphi_\xi^s(v)
$$
is parallel to $v$ and hence
$$
\left.\frac{d}{ds}\right|_{s =0} f(s^2) = [\xi, \eta](v) \in E^p(v).
$$
Therefore, each leaf of $E^p$
 though $v$ is a $k$-dimensional submanifold of $SX$, given by $P(v)$.
 Now the flat strip theorem (see \cite{Es}) implies that the projections $F(v) := \pi(P(v))$ are $k$-flats, i.e., totally geodesic flat spaces
isometric to the Euclidean space $\mathbb{R}^k$. Therefore, the unit tangent bundle of a $k$-flat
is flow invariant and as one easily checks tangent to the central distribution $E^c$.\\
Since $\phi^t$ defines a partially hyperbolic flow
the stable and unstable distributions $E^s$ and $E^u$ are integrable and tangent to $W^s$ and $W^u$.

\end{proof}

Let $X$ be a simply connected manifold without conjugate points. Let $B(q,t)$ be the open ball of radius $t$ about
$q$. For $q_1, q_2 \in X \setminus B(q,t)$ we call
\begin{eqnarray*}
d^q_t(q_1, q_2) &:= &\inf \{ L(\gamma) \mid  \;
\gamma:[a,b] \to  X \setminus B(q,t) \\
&& \text{ piecewise smooth curve
joining } \; q_1 \text{ and } q_2 \}.
\end{eqnarray*}

The following lemma is important for the proof of theorem~\ref{5C}.
\begin{lemma}\label{cf}
Let $X$ be a simply connected manifold without focal points which fulfills the assumption of theorem~\ref{5C}
and assume that its rank is at least 2. Given  $v \in S_qX$ then for each $w \in S_qX$ the following is equivalent
\begin{enumerate}
\item[(a)] $w \in  S_q F(v)$,
\item[(b)] $
\displaystyle\varlimsup_{t\to \infty} \frac{1}{t}d^q_{t}(c_v(t) , c_w(t)) < \infty,
$
\item[(c)] there exists $\rho \ge 0$ such that
$
\displaystyle\varlimsup_{t\to \infty} \frac{1}{t}d^q_{t -\rho}(c_v(t) , c_w(t)) < \infty.
$
\end{enumerate}
\end{lemma}
\begin{proof}
We can assume that $X$ is not flat.
Assume $w \in  S_q F(v)$. Consider a shortest curve $x: [0, a] \to S_q F(v)$ such that
$x(0) = v$, $x(a) = w$ and $\| x'(s) \| = 1$. In particular, $x'(s) \in \cal L(x(s))$. Then
$\gamma_t(s) = \exp_q(t x(s))$ connects $c_v(t)$ and  $ c_w(t)$ and the image is in the complement
of $B(q,t)$. Furthermore, $\dot \gamma_t(s)= J(t)$ is the Jacobi field along the geodesic $c_{x(s)}$
with $J(0) = 0 $ and $J'(0) = x'(s)$. Hence, $J(t) = A_{x(s)}(t)(x'(s))_t$ where
$A_{x(s)}$ is the Jacobi tensor along  $c_{x(s)}$ with $A_{x(s)}(0)$
$ = 0$ and  $A_{x(s)}'(0) = \id$.
Since $x'(s) \in \cal L(x(s))$ and $\| x'(s) \| = 1$ proposition~\ref{5B} implies
$$
\| \dot \gamma_t(s) \| = \| A_{x(s)}(t)(x'(s))_t \| = t
$$
Hence, $L(\gamma_t) = ta$ and, therefore, (a) implies (b).\\

Since for all $\rho \ge 0$ and $q_1, q_2 \in X \setminus B(q,t)$ we have
$d^q_{t -\rho}(q_1, q_2) \le d^q_t(q_1, q_2)$, assertion (c) follows from (b).\\

Now assume that (c) holds and $w \notin  S_q F(v)$. Let $\gamma:[0, 1] \to X$ be a smooth
curve such that $\gamma(0) = v$, $\gamma(1) = w$ and $d( \gamma(s) , q) = f(s) \ge t - \rho$.
Consider the curve $x: [0,1] \to S_qX$ such that $\gamma (s) = \exp_q(f(s) x(s))$ and the geodesic variation
$$
\varphi(s,u) = \exp_q(u x(s)), \; \; \text{where}  \; \; 0 \le u \le f(s).
$$
Then
$$
\frac{\partial}{\partial s} \varphi(s,u) = J_{x(s)}(u)
$$
is the perpendicular Jacobi field along $c_{x(s)}(u)$ with initial conditions
$J_{x(s)}(0) =0$ and $J'_{x(s)}(0) = x'(s)$. Since
\begin{eqnarray*}
\left.\frac{\partial}{\partial s} \right|_{s=s_0} \gamma(s) & =&
\left.\frac{\partial}{\partial s} \right|_{s=s_0} \varphi(s, f(s_0)) +
\left.\frac{\partial}{\partial u}\right|_{u=f(s_0)} \varphi(s_0,u) f'(s_0)\\
& =& J_{x(s_0)}(f(s_0)) + \dot c_{x(s_0)}(f(s_0)) f'(s_0)
\end{eqnarray*}
the estimate
$$
\| \left.\frac{\partial}{\partial s} \right|_{s=s_0} \gamma(s) \| \ge  \|J_{x(s_0)}(f(s_0))\|
$$
holds.  As above we have $J_{x(s)}(u) = A_{x(s)}(u)(x'(s))_u$. Decompose
$ x'(s) = y(s) + z(s)$, where $y(s) \in \cal L(x(s)$ and $z(s) \in \cal L^\perp(x(s)$. Since
$w \notin  S_q F(v)$ there is a constant $b >0$ such
$\int \limits_{0}^1 \|z(s)\| ds \ge b$. Using proposition~\ref{5B}, we obtain
$$
L(\gamma) = \int \limits_{0}^1 \| \dot \gamma(s) \| ds \ge  a' e^{\alpha( t -\rho)} \int \limits_{0}^1 \|z(s)\| ds \ge  a' e^{\alpha( t -\rho)} b
$$
for all $t >0$ in contradiction to (c).
\end{proof}

{\it Proof of theorem \ref{5C}}\\
Assume that the rank of X is at least 2. Consider $v \in S_pX$ and $v' \in W^s(v)$ where $v' \not= v$ and $\pi(v') =q$. This implies that
$d(c_v(t) , c_{v'}(t)) $ converges to $0$ as $t$ tends to $\infty$. For each $w \in S_pF(v)$  define $w' = - \grad b_w(q)$. Since $X$ has no
focal points
$d(c_{w'}(t), c_w(t)) \le d(p,q)$. In particular for  $\rho = d(p,q)$ we obtain
\begin{eqnarray*}
\frac{1}{t}d^p_{t -\rho} (c_{w'}(t), c_{v'}(t)) &\le&
\frac{1}{t}\Big( d^p_{t -\rho}(  c_{w'}(t), c_{w}(t) ) + d^p_{t -\rho}(c_{w}(t), c_{v}(t)) \\
&& +  d^p_{t -\rho}( c_{v}(t), c_{v'}(t)) \Big)
\end{eqnarray*}
is bounded for $t \ge 0$. Using Lemma~\ref{cf} this implies $w' \in S_qF(v')$. Furthermore,
$$
\theta':= \sphericalangle_q(w', v') = \sphericalangle_p(w,v) =: \theta.
$$
where $\theta, \theta' \in [- \pi, \pi]$.
To see this we note that $c_v$ and $c_w$ respectively $c_{v'}$ and $c_{w'}$ are lying in the Euclidean spaces $F(v)$ resp. $F(v')$. Therefore, we have
$$
d(c_v(t), c_w(t)) = 2t \sin( \frac{\theta}{2}) \; \; \text{and} \;\; d(c_{v'}(t), c_{w'}(t)) = 2t \sin( \frac{\theta'}{2})
$$
and the triangle inequality implies
\begin{eqnarray*}
2t | \sin( \frac{\theta}{2})-  \sin( \frac{\theta'}{2}) | & =&   |  d(c_v(t), c_w(t) - d(c_{v'}(t), c_{w'}(t)) | \\
& \le& | d(c_v(t) , c_{v'}(t))+ d(c_{w'}(t), c_w(t)) | \le A
\end{eqnarray*}
for some constant $A >0$ and for all $t \ge 0$. Hence $\theta = \theta'$.
In particular for $w = -v$ we obtain
$$
w' = (-v)' = -(v')
$$
since
$$
 \sphericalangle_q((-v)', v') = \sphericalangle_p(-v ,v) = \pi.
$$
Consequently
$$
 d(c_{v'}(-t), c_v(-t))= d(c_{-(v')}(t) , c_{-v}(t)) = d(c_{(-v)'}(t) , c_{-v}(t) \le d(p,q)
$$
for all $t \ge 0$.
Since $ d(c_v(t) , c_{v'}(t))$ converges to $0$ for $ t \to \infty$ we have
$v= v'$ which is in contradiction to the assumption.
 Hence the rank of
$X$ must be one.

\section{Appendix}
In this appendix we collect properties of $(1,1)$-tensors which are used in this paper.
\begin{dfn}
Let $B: I \to \End{Nc} $ be a $(1,1)$-tensor along a geodesic $c: I \to M$. Then for $t_0, t \in I$
we define
$
 \int\limits_{t_0}^{t} B(s) ds :I \to \End{Nc} $
 via
 $$
\left\langle \int\limits_{t_0}^{t} B(s) ds \ x_t, y_t \right \rangle :=
 \int\limits_{t_0}^{t}  \langle B(s)  x_s, y_s \rangle ds,
 $$
 where $x_t, y_t$ are orthogonal parallel vector fields along $c$.
\end{dfn}
\begin{remark} If $B: I \to \End{Nc} $ is a symmetric, positive definite $(1,1)$-tensor along a geodesic $c: I \to M$, then  $
 \int\limits_{t_0}^{t} B(s) ds :I \to \End{Nc} $  symmetric, positive definite as well.
 Furthermore, the following estimates hold
 \begin{eqnarray}
 \left\| \int\limits_{t_0}^{t} B(s) ds  \right\| &\le& \int\limits_{t_0}^{t} \|B(s) \|ds \\
  \left\| \left(\int\limits_{t_0}^{t} B(s) \right)^{-1} ds \right \|^{-1}
  &\ge&  \int\limits_{t_0}^{t}\|B(s) ^{-1}\|^{-1} ds \label{symtensor}
  \end{eqnarray}
 \end{remark}
\begin{prop} \label{jac2}
Let $M$ be a Riemannian manifold.
Let $Y: I \to \End{Nc} $ be a Lagrange tensor along a geodesic $c: I \to M$ which is nonsingular
for all $t \in I$. Then for $t_0 \in I$ and any other Jacobi tensor $Z$ along $c$, there exists constant
tensors $C_1$ and $C_2$ such that
$$
Z(t) = Y(t) \left( \int\limits_{t_0}^{t}( Y^{\ast} Y)^{-1}(s) ds \ C_1 + C_2 \right)
$$
for all $t \in I$.

\end{prop}
\begin{proof}
Since $Y$ is nonsingular there exists a $(1,1)$-tensor $B$ along $C$ such that
$$
Z(t) = Y(t) B(t).
$$
Differentiating this equation twice yields
$$
Z''(t) = Y''(t) B(t) + 2 Y'(t) B'(t) + Y(t) B''(t).
$$
Since $Y$ and $Z$ are Jacobi tensors, we obtain
$$
Z''(t) = -R(t)Z(t) = R(t) Y(t) B(t) \quad \text{and} \quad Y''(t) = -R(t)Y(t)
$$
and, therefore,
$$
2 Y'(t) B'(t) + Y(t) B''(t) = 0.
$$
Since $Y$ is nonsingular $B'$ is a solution of the differential equation
\begin{equation} \label{eqjac}
2Y^{-1}(t) Y'(t) G(t) + G'(t) = 0.
\end{equation}
Since $Y$ is a Lagrange tensor we have
$$
W(Y, Y) = Y^{\ast '} (t) Y(t) - Y^{\ast}(t) Y'(t) = 0.
$$
Therefore, $G(t) := (Y^{\ast} Y)^{-1}(t)$
is a solution of \eqref{eqjac} as the following computation shows
\begin{eqnarray*}
G'(t) &= &\left( (Y^{\ast} Y)^{-1}) \right)'(t) = (Y^{\ast} Y)^{-1}(t) (Y^{\ast} Y)'(t) (Y^{\ast} Y)^{-1}(t) \\
&=& - Y^{-1}(t) Y^{\ast^{-1}}(t) \left( Y^{\ast'}(t)Y(t) + Y^{\ast}(t) Y'(t) \right)Y^{-1}(t)  Y^{\ast^{-1}}(t)\\
&=& -2  Y^{-1}(t) Y^{\ast^{-1}}(t) Y^{\ast}(t) Y'(t)Y^{-1}(t)  Y^{\ast^{-1}}(t)\\
&=& -2 Y^{-1}(t)Y'(t) G(t).
\end{eqnarray*}
Hence, an arbitrary solution of \eqref{eqjac} is of the form $G(t) = (Y^{\ast} Y)^{-1}(t) C$, where
$C$ is a constant tensor along $c$.
Therefore, $B'(t) = (Y^{\ast} Y)^{-1}(t) C_1$ and integration yields
$$
B(t) = \int\limits_{t_0}^{t}\left(Y^{\ast} Y \right)^{-1}(s) ds \ C_1 + C_2
$$
for constant tensors $C_1, C_2$ along $c$.
\end{proof}

The following properties of the stable and unstable Jacobi tensors $U_v$ and $S_v$ defined in section 2 are
 frequently used in this paper.

\begin{lemma} \label{central}
Let $M$ be a manifold without conjugate points. Then for all $v \in SM$ we have
\begin{equation}\label{c1}
S_{\phi^u (v)}(t) = S_v(t+u) S_v^{-1}(u) \; \; \text{and} \; \;U_{\phi^u (v)}(t) = U_v(t+u) U_v^{-1}(u)
\end{equation}
\begin{equation}\label{c2}
S _{\phi^t(v)}'(0) =  S_v'(t) S_v^{-1}(t)) \; \; \text{and} \; \; U_{\phi^t(v)}'(0)=U_v'(t) U_v^{-1}(t)
\end{equation}
\begin{eqnarray}\label{c3}
U_{\phi^t(v)}'(0) - S _{\phi^t(v)}'(0) &=&U_v^{{\ast}^{-1}}(t)(U_v'(0) -S_v'(0)) S_v^{-1}(t)\\
&=&S_v^{\ast -1}(t)(U_v'(0) -S_v'(0)) U_v^{-1}(t). \nonumber
\end{eqnarray}
Furthermore,
\begin{equation}\label{c4}
U_{\phi^t(v)}'(0) - S _{\phi^t(v)}'(0) = S^{\ast-1}_v(t) \left(\int\limits_{-\infty}^{t} (S_v^{\ast}S_v)^{-1}(u) du \right)^{-1} S_v^{-1}(t).
\end{equation}
Furthermore, for all $x \in \cal L(v) $ we have $S_v(t)x_t = U_v(t) x_t \in  \cal L(\phi^tv) $.
Moreover, $S_v(t)y = U_v(t)y \in \cal L(\phi^tv)$ for all $y \in \cal L(\phi^tv)$.
\end{lemma}
\begin{proof}
Let $S_{v,r}(t)$ the Jacobi tensor  along the geodesic $c_v(t)$ with $S_{v,r}(0) = \id$ and
$S_{v,r}(r) = 0$. Then
$$
S_{\phi_u v,r}(t) = S_{v, r +u}(t+u) S_{v, r+u}^{-1}(u),
$$
since both sides define for fixed $u $ and $r$  Jacobi tensor, which
agree at $t=0$ and $t = r$. Taking the limit $r \to \infty$ yields the first assertion
of \eqref{c1}. The second follows with a similar argument.
Differentiating this relations yields \eqref{c2}.\\
To prove \eqref{c3} consider
the Wronskian $W(U_v, S_v)(t) $ given by
$$
W(U_v, S_v)(t) = U_v^{\ast'}(t) S_v(t) - U_v^{\ast}(t) S_v'(t) = W(U_v, S_v)(0) =B_v,
$$
where
$$
B_v = U_v'(0) -S_v'(0).
$$
This yields
$$
(U_v'(t) U_v^{-1}(t))^{\ast}- S_v'(t) S_v^{-1}(t) = U_v^{\ast -1}(t) B_v S_v^{-1}(t).
$$
Since
$$
(U_v'(t) U_v^{-1}(t))= U_{\phi^t(v)}'(0) \; \; \text{and} \; \; S_v'(t) S_v^{-1}(t))= S_{\phi^t(v)}'(0)
$$
are symmetric, we obtain the first identity and taking the adjoint we obtain the second one.\\
 To prove \eqref{c4} consider $ 0 <r, s$ and  $t \in (-r, s)$.
Then using \ref{jac1}
 there exists a constant tensor $K_{s,r}$ such that
$$
U_{v,r}(t) = S_{v,s}(t) \int\limits_{-r}^{t} (S_{v,s}^{\ast}S_{v,s})^{-1}(u) du K_{s,r}.
$$
Evaluating and differentiating this identity at $t= 0$ we obtain the equation
$$
U_{v,r}(t) = S_{v,s}(t) \int\limits_{-r}^{t} (S_{v,s}^{\ast}S_{v,s})^{-1}(u) du (U_{v,r}'(0) - S_{v,s}'(0)).
$$
Furthermore, for all $t < s$ the limit
$$
\lim\limits_{r \to \infty}\int\limits_{-r}^{t} (S_{v,s}^{\ast}S_{v,s})^{-1}(u) du = \int\limits_{-\infty}^{t} (S_{v,s}^{\ast}S_{v,s})^{-1}(u) du
$$
exists and is invertible. Hence,
\begin{equation} \label{eq1}
\left(\int\limits_{-\infty}^{t} (S_{v,s}^{\ast}S_{v,s})^{-1}(u) du \right)^{-1} S_{v,s}(t)^{-1}U_v(t)=
U_v'(0) - S_{v,s}'(0).
\end{equation}
Passing to the limit $s \to \infty$, we obtain
 $$
\left(\int\limits_{-\infty}^{t} (S_{v}^{\ast}S_{v})^{-1}(u) du \right)^{-1} S_{v}(t)^{-1}U_v(t)=
U_v'(0) - S_{v}'(0).
$$

Inserting \eqref{eq1} into the second identity of \eqref{c3}, we obtain
$$
U_{\phi^t(v)}'(0) - S _{\phi^t(v)}'(0) = S^{\ast-1}_v(t) \left(\int\limits_{-\infty}^{t} (S_v^{\ast}S_v)^{-1}(u) du \right)^{-1} S_v^{-1}(t).
$$
Finally let $x \in \cal L(v)$ . i.e., $U_v'(0)x = S_{v}'(0)x$. Then, $U_v(t)x_t = S_v(t)x_t $ and
$U'_v(t)x_t = S'_v(t)x_t$. This implies
$$
U_{\phi^t(v)}'(0) U_v(t)x_t   = U_v'(t) U_v^{-1}(t)U_v(t)x_t  = S_v'(t)x_t= S_{\phi^t(v)}'(0) S_v(t)x_t
$$
and therefore, $ U_v(t)x_t = S_v(t)x_t  \in \cal L(\phi^t(v))$.
\end{proof}

{\it Acknowledgement.} I would like to thank Norbert Peyerimhoff for asking the question
on the volume growth of noncompact harmonic manifolds formulated in the introduction, which initiated this work. Furthermore,
I am grateful to him for some suggestions improving
the presentation of this article.
%%%%%%%%%%%%%%%%%%%%%%%%%%%%%%%%%%%%%%%%%%%%%%%%%%%%%%%%%%%%%%%%%%%%%%

%%%%%%%%%%%%%%%%%%%%%%%%%%%%%%%%%%%%%%%%%%%%%%%%%%%%%%%%%%%%%%%%%%%%%%


\begin{thebibliography}{ZZ99}
\bibitem[BBE]{BBE}
W.~Ballmann, M.~Brin, P.~Eberlein.
\textit{Structure of manifolds of nonpositive curvature I},
Ann. of Math. (2) \textbf{122} (1985), 171--203.

\bibitem[BCG]{BCG}
G.~Besson, G.~Courtois, S.~Gallot.
\textit{Entropies et regidit\'{e}s des espaces
localement sym\'{e}triques de courbure strictement n\'{e}gative.},
Geom. Funct. Anal. \textbf{5} (1995), no. 5, 731--799.

\bibitem[BFL]{BFL}
Y.~Benoist, P.~Foulon, F.~Labourie.
\textit{Flots d' Anosov \`{a} distributions stable et instable differ\'{e}ntiables},
J. Amer. Math. Soc.  \textbf{5} (1992), no. 1, 33--74.

\bibitem[Be]{Be}
A.L.~Besse.
\textit{Manifolds all of whose geodesics are closed},
Ergebnisse der Mathematik und ihrer Grenzgebiete [Results in Mathematics and Related Areas]
\textbf{93}, Springer Verlag, Berlin-New York, 1978.

\bibitem[BH]{BH}
M.R.~Bridson, A.~Haefliger.
\textit{Metric spaces of nonpositive curvature}.
Grundlehren der mathe\-matischen Wissenschaften, [Fundamental Principles of Mathematical Sciences], \textbf{319},
Springer-Verlag, Berlin, 1999. xxii + 643 pp.

\bibitem[Bo]{Bo}
J.~Bolton.
\textit{Conditions under which a geodesic flow is Anosov},
Math. Ann. \textbf{240} (1979), no. 2, 103--113.

\bibitem[DR]{DR}
E.~Damek, F.~Ricci.
\textit{A class of nonsymmetric harmonic Riemannian spaces},
Bull. Amer. Math. Soc. (N.S.) \textbf{27} (1992), no. 1, 139--142.

\bibitem[Eb]{Eb}
P.~Eberlein.
\textit{When ist the geodesic flow of Anosov type? I.},
J. Differential Geometry \textbf{8} (1973), 437--463.

\bibitem[EO]{EO}
P.~Eberlein, B.~O'Neill.
\textit{Visibility manifolds},
Pacific J. Math. \textbf{46} (1973), 45--109.

\bibitem[Es]{Es}
J.H.~Eschenburg.
\textit{Horospheres and the stable part of the geodesic flow},
Math. Zeitschrift \textbf{153} (1977), no. 3, 237--251.

\bibitem[FM]{FM}
A.~Freir\'{e}, R.~Ma\~{n}\'{e}.
\textit{On the entropy of the geodesic flow in manifolds without conjugate points},
Invent. Math. \textbf{69} (1982), no. 3, 375--392.

\bibitem[FL]{FL}
P.~Foulon, F.~Labourie.
\textit{Sur les vari\'{e}t\'{e}s compactes asymptotiquement harmoniques},
Invent. Math. \textbf{109} (1992), no. 1, 97--111.

\bibitem[Gr]{Gr}
M.~Gromov.
\textit{Hyperbolic groups},
In: Essays in group theory, Math. Sci. Res. Inst. Publ., \textbf{8}, Springer, New York (1987), 75--263.

\bibitem[He]{He}
J.~Heber.
\textit{On harmonic and asymptotically harmonic homogeneous spaces},
Geom. Funct. Anal. \textbf{16} (2006), no. 4, 869--890.

\bibitem[Kl]{Kl}
W.~Klingenberg.
\textit{Riemannian manifolds with geodesic flow of Anosov type},
Ann. of Math. (2) \textbf{99} (1974), 1--13.

\bibitem[Kn]{Kn}
G.~Knieper.
\textit{Hyperbolic Dynamics and Riemannian Geometry},
in Handbook of Dynamical Systems, Vol. 1A 2002, Elsevier Science B.,
eds. B.~Hasselblatt and A.~Katok, (2002), 453--545.

\bibitem[Kn1]{Kn1}
G.~Knieper.
\textit{On the asymptotic geometry of nonpositively curved manifolds},
Geom. Funct. Anal. \textbf{7} (1997), no. 4, 755--782.

\bibitem[Li]{Li}
A.~Lichnerowicz.
\textit{Sur les espaces riemanniens compl\`{e}ment harmoniques} (French),
Bull. Soc. Math. France \textbf{72} (1944), 146--168.

\bibitem[Ma]{Ma}
R.~Man\'{e}.
\textit{On a theorem of Klingenberg},
Dynamical Systems and Bifurcation Theory, (Rio de Janeiro, 1985), M.~Camacho, M.~Pacifico and
F.~Takens eds., Pitman Res. Notes Math. \textbf{160} Longman Sci. Tech., Harlow (1987), 319--345.

\bibitem[Ni]{Ni}
Y.~Nikolayevsky.
\textit{Two theorems on harmonic  manifolds},
Comment. Math. Helv. \textbf{80} (2005), no. 1, 29--50.

\bibitem[RSh1]{RSh1}
A.~Ranjan, H.~Shah.
\textit{Convexity of spheres in a manifold without conjugate points},
Proc. Indian Acad. Sci. Math. Sci. \textbf{112}  (2002), no. 4, 595--599.

\bibitem[RSh2]{RSh2}
A.~Ranjan, H.~Shah.
\textit{Harmonic manifolds with minimal horospheres},
J. Geom. Anal. \textbf{12} (2002), no. 4, 683--694.

\bibitem[RSh3]{RSh3}
A.~Ranjan, H.~Shah.
\textit{Busemann functions in a harmonic manifold},
Geom. Dedicata \textbf{101} (2003), 167--183.

\bibitem[SW]{SW}
N.E.~Steenrod,  J.H.C.~Whitehead.
\textit{Vector fields on the $n$-Sphere},
Proc. Nat. Acad. Sci. U. S. A. \textbf{37} (1951), 58--63.

\bibitem[Sz]{Sz}
Z.I.~Szab\'{o}.
\textit{The Lichnerowicz Conjecture on Harmonic manifolds},
J. Differential Geom. \textbf{31} (1990), no. 1, 1--28.
\end{thebibliography}
\end{document}